\documentclass{amsart}
\usepackage{amssymb, bm, amsmath, latexsym, mathabx, dsfont,tikz}
\usepackage{makecell, multirow}
\usepackage{enumitem}

\renewcommand{\baselinestretch}{\baselinestretch}
\renewcommand{\baselinestretch}{1.1}
\numberwithin{equation}{section}

\newtheorem{thm}{Theorem}[section]
\newtheorem{lem}[thm]{Lemma}
\newtheorem{cor}[thm]{Corollary}

\theoremstyle{definition}

\theoremstyle{remark}
\newtheorem{rmk}[thm]{Remark}

\numberwithin{equation}{section}

\newcommand{\pra}{{\, \xrightarrow{\nn} \,}}

\newcommand{\ra}{{\, \rightarrow \,}}

\newcommand{\gen}{\text{gen}}

\newcommand{\z}{{\mathbb Z}}
\newcommand{\nn}{{\mathbb{N}_0}}
\newcommand{\q}{{\mathbb Q}}

\newcommand{\n}{{\mathbb N}}

\newcommand{\aaa}{{\bm{a}}}

\newcommand{\Mod}[1]{\ (\mathrm{mod}\ #1)}

\newcommand{\df}[1]{\langle #1 \rangle}

\newenvironment{newenum}
{\begin{enumerate}[label={\rm(\arabic*)}]}
	{\end{enumerate}}

\begin{document}


\author{Jangwon Ju}
\address{Department of Mathematics, University of Ulsan,
 Ulsan 44610, Republic of Korea}
\email{jangwonju@ulsan.ac.kr}

\author{Daejun Kim}
\address{Research Institute of Mathematics, Seoul National University,
	Seoul 08826, Republic of Korea}
\email{goodkdj@snu.ac.kr}

\thanks{This work of the first author was supported by the National Research Foundation of Korea(NRF) grant funded by the Korea government(MSIT) 
(NRF-2019R1F1A1064037)}

\subjclass[2010]{11E12, 11E25, 11E20}

\keywords{Universal sums of pentagonal numbers, Generalized Cauchy's lemma}

\thanks{}


\title[The pentagonal theorem of $63$]{The pentagonal theorem of sixty-three and generalizations of Cauchy's lemma}

\begin{abstract} 
	In this article, we study the representability of integers as sums of pentagonal numbers, where a pentagonal number is an integer of the form $P_5(x)=\frac{3x^2-x}{2}$ for some non-negative integer $x$.
	In particular, we prove the ``pentagonal theorem of $63$", which states that a sum of pentagonal numbers represents every non-negative integer if and only if it represents the integers
	\begin{center}
	$1$, $2$, $3$, $4$, $6$, $7$, $8$, $9$, $11$, $13$, $14$, $17$, $18$, $19$, $23$, $28$, $31$, $33$, $34$, $39$, $42$, and $63$.
	\end{center}
	We also introduce a method to obtain a generalized version of Cauchy's lemma using representations of binary integral quadratic forms by quaternary quadratic forms, which plays a crucial role in proving the results.
\end{abstract}

\maketitle

\section{Introduction}

The history on the representations of non-negative integers as sums of polygonal numbers dates back to 1638, as Fermat asserted that every non-negative integer is written as a sum of at most $m$ $m$-gonal numbers.
This was proved by Lagrange in 1770, Gauss in 1796, and Cauchy in 1815 in the cases when $m=4$, $m=3$, and $m\ge5$, respectively.
More precisely, for any positive integer $m\ge3$, an integer of the form 
$P_m(x)=\frac{m-2}{2}(x^2-x)+x$
for some $x\in\nn:=\n\cup\{0\}$ is called a {\em polygonal number of order $m$}, or simply an {\em $m$-gonal number}.
For an $\aaa=(a_1,\ldots,a_k) \in \n^k$, and an $\bm{x}=(x_1,\ldots,x_k)\in \n_0^k$, we define the sum
$$
P_{m,\aaa}(\bm{x}):=\sum_{i=1}^k a_i P_m(x_i).
$$
In general, the sum $P_{m,\aaa}$ is said to be {\em universal} over $\nn$ if for any non-negative integer $N$, the diophantine equation
\begin{equation}\label{eqPrepN}
N=P_{m,\aaa}(\bm{x})
\end{equation}
has a solution $\bm{x}\in \n_0^k$. If Equation \eqref{eqPrepN} is solvable over $\nn$, we say the sum $P_{m,\aaa}$ {\em represents} an integer $N$ over $\nn$, and write $N\pra P_{m,\aaa}$.
Fermat's claim may then be restated as for $\aaa=(1,\ldots,1)$ of length $m$, the sum $P_{m,\aaa}$ is {\em universal} over $\nn$.

The Fermat's polygonal number theorem was generalized in many directions.
In 1862, Liouville generalized the Gauss's theorem by proving that for $\aaa=(a_1,a_2,a_3)\in\n^3$ with $a_1\le a_2\le a_3$, the sum $P_{3,\aaa}$ is universal over $\nn$ if and only if $\aaa$ is one of the following:
$$
(1,1,1), \ (1,1,2), \ (1,1,4),\ (1,1,5),\ (1,2,2) ,\ (1,2,3) ,\ (1,2,4).
$$
An analogous generalization of Lagrange's four square theorem was concerned by Ramanujan in 1917 and was completed by Dickson \cite{Di1}. 
Moreover, one may easily observe from the $15$-Theorem of Conway and Schneeberger, which was elegantly reproved by Bhargava \cite{B}, that for any $\aaa\in \n^4$ the sum $P_{4,\aaa}$ is universal over $\nn$ if and only if it represents the integers $1$, $2$, $3$, $5$, $6$, $7$, $10$, $14$, and $15$ over $\nn$. 

We refer the readers to Duke's survey paper \cite{Du} for a more complete history of related questions on sums of polygonal numbers, and Conway's paper \cite{C} for the history of works on universal quadratic forms from Lagrange's to the $290$-conjecture, which is now the $290$-Theorem as was proved by Bhargava and Hanke \cite{BH}. 

Recently, in 2013, Bosma and Kane \cite{BK} proved that for any $\aaa\in\n^k$, the sum $P_{3,\aaa}$ is universal over $\nn$ if and only if it represents the integers $1$, $2$, $4$, $5$, and $8$ over $\nn$.
The aim of this article is to prove the following ``pentagonal theorem of $63$" on the sums of pentagonal numbers:

\begin{thm}\label{63thm}
The sum $P_{5,\aaa}$ of pentagonal numbers is universal over $\nn$ if and only if it represents the integers 
$$\text{$1$, $2$, $3$, $4$, $6$, $7$, $8$, $9$, $11$, $13$, $14$, $17$, $18$, $19$, $23$, $28$, $31$, $33$, $34$, $39$, $42$, and $63$.}$$
\end{thm}

Note that if $P_{5,\aaa}$ is universal over $\nn$ for an $\aaa=(a_1,\ldots,a_k)\in \n^k$, then so is $P_{5,\aaa'}$ for any $\aaa'\in\n^k$ obtained by some permutation on the components of $\aaa$. Hence we may assume without loss of generality that 
$$
a_1\le a_2\le  \cdots \le a_k.
$$
Using the escalator method introduced by Bhargava in \cite{B}, it is not hard to conclude that in order for $P_{5,\aaa}$ with $\aaa\in\n^k$ to be universal over $\nn$, we necessarily have $k\ge 4$ and  $(a_1,a_2,a_3,a_4)\in\mathcal{A}'$, where $\mathcal{A}'$ is the set consisting of the following vectors:
$$
\begin{array}{lll}
(1,1,1,a_4) \text{ with }  1\le a_4 \le 4, & (1,1,2,a_4) \text{ with }  2\le a_4 \le 9,\\
(1,1,3,a_4) \text{ with }  3\le a_4 \le 7, & (1,2,2,a_4) \text{ with }  2\le a_4 \le 6,\\
(1,2,3,a_4) \text{ with }  3\le a_4 \le 9, & (1,2,4,a_4) \text{ with }  4\le a_4 \le 8.
\end{array}
$$
Theorem \ref{63thm} will be established directly once we determine the set 
of all integers that are represented by $P_{5,\aaa}$ for each vector $\aaa$ in the set $\mathcal{A}=\mathcal{A}'\cup \{(1,2,4,12)\}$.
We refer the readers to Section \ref{sec-escalation} for this escalation argument in detail.
Let 
$$E(P_{5,\aaa}):= \nn \setminus \{ N \in \nn \mid N\pra P_{5,\aaa}\}$$
denote the set of all non-negative integers not represented by $P_{5,\aaa}$ over $\nn$. 
From the above viewpoint, the following theorem is the most crucial part of this article.

\begin{thm}\label{thmquatpolyrep}
Let $\aaa\in \mathcal{A}$ be a vector listed in the above, and let $N_\aaa$ be a constant given in Table \ref{table-data1} or \ref{table-data2}. 
\begin{enumerate}[label={\rm(\arabic*)},leftmargin=*]
\item If $\aaa\neq (1,2,4,5)$, then we have $N\pra P_{5,\aaa}$ for any integer $N$ with $N\ge N_\aaa$.
In particular, the set $E(P_{5,\aaa})$ of all non-negative integers not represented by $P_{5,\aaa}$ over $\nn$ is determined as  in Table \ref{table-data1} or \ref{table-data2}.
\item If $\aaa=(1,2,4,5)$, then we have $N\pra P_{5,\aaa}$ for any integer $N$ with $N\ge N_\aaa$ and $N \Mod{12} \in \{0,3,4,6,7,8,9,11\}$.
\end{enumerate}

\end{thm}

\begin{rmk}\label{rmk1}{\color{white} a}
\begin{enumerate}[label={\rm(\arabic*)},leftmargin=*]

\item\label{rmk1:1} For each $\aaa\in \mathcal{A}$ except $(1,2,4,5)$, the set $E(P_{5,\aaa})$ can be determined by checking directly whether or not $N\pra P_{5,\aaa}$ for any $N\le N_\aaa$.
On the other hand, we conjecture that $E(P_{5,(1,2,4,5)})=\{13\}$. We have verified this conjecture up to $10^7$, that is, we have checked $N\pra P_{5,(1,2,4,5)}$ for $N\neq 13$ up to $10^7$.

\item
Theorem \ref{thmquatpolyrep} confirms \cite[Conjecture 5.2 (ii)]{S} of Sun,   that is, the sum $P_{5,(1,b,c,d)}$ is universal over $\nn$ if $(b,c,d)$ is among the following $15$ triples:
	$$
	\begin{array}{l}
	(1,1,2),\ (1,2,2),\ (1,2,3),\ (1,2,4),\ (1,2,5),\ (1,2,6),\ (1,3,6), \\
	(2,2,4),\ (2,2,6),\ (2,3,4),\ (2,3,5),\ (2,3,7),\ (2,4,6),\ (2,4,7),\ (2,4,8).
	\end{array}
	$$
Indeed, the universality of $P_{5,(1,b,c,d)}$ over $\nn$ was already proved, for
$(b,c,d)=(1,2,2)$ and $(1,2,4)$ by Meng and Sun \cite{MS}, and for 
$(b,c,d)=(1,1,2)$, $(1,2,3)$, $(1,2,6)$, and $(2,3,4)$ by  Krachun and Sun \cite{KS}.
For each of these six vectors $\aaa=(1,b,c,d)$, they showed that $N\pra P_{5,\aaa}$ for any $N\ge N_\aaa'$ and checked the representability directly for $N\le N_\aaa'$.
We note that our method gives dramatic improvements on such a ``theoretical bound" $N_\aaa$.
For example, for $\aaa=(1,1,2,6)$, $N_\aaa=73138$ while $N_\aaa'=897099189$.
%
%
\end{enumerate}
\end{rmk}


The proof of Theorem \ref{thmquatpolyrep} for each $\aaa\in\mathcal{A}$ mainly uses a generalized version of Cauchy's lemma, of which the proof relies on finding a special representation of binary quadratic forms by quaternary quadratic forms.
We first introduce, so-called, Cauchy's lemma, which Cauchy himself used to prove the general case of Fermat's polygonal number theorem (see, for example, \cite{N1} for a proof of Cauchy's lemma and related results).\\[5pt]
\noindent {\bf Cauchy's Lemma.}
{\em
Let $a$ and $b$ be odd positive integers such that 
$$b^2<4a \quad  \text{and} \quad 3a<b^2+2b+4.$$ 
Then there exist non-negative integers $s,t,u,v$ such that}
$$
\begin{cases}
a=s^2+t^2+u^2+v^2,\\
b=s+t+u+v.
\end{cases}
$$

Let $\aaa=(a_1,a_2,a_3,a_4)\in \n^4$ with $a_1 \le a_2\le a_3\le a_4$ and let $A_\aaa=\sum_{i=1}^4 a_i$. 
In Section \ref{preliminaries} and Section \ref{sec-lem}, we develop a method to obtain a set $\mathfrak{S}_s\subset (\z/s\z)^2$ ($s\in\n$) such that the following generalization of Cauchy's lemma holds.
This is a consequence of Lemmas \ref{posrep}, \ref{eq1-latticerep}, and \ref{rep-gen-to-lattice}.\\[5pt]
\noindent {\bf Generalized Cauchy's Lemma.}
{\em Let $a$ and $b$ be positive integers such that 
	\begin{enumerate}[label={\rm(\arabic*)}]
		\item $a\equiv b\Mod{2}$ and $b^2<A_\aaa a$,
		\item $a$ and $b$ satisfy some local condition corresponding to $\aaa$,
		\item $(a,b)\Mod{s} \in \mathfrak{S}_s$,
		\item $(A_\aaa-a_1)a \le b^2$.
	\end{enumerate}
	Then there exist non-negative integers $x_1,x_2,x_3,x_4$ such that}
\begin{equation}\label{eq0}
\begin{cases}
a=a_1x_1^2+a_2x_2^2+a_3x_3^2+a_4x_4^2,\\
b=a_1x_1+a_2x_2+a_3x_3+a_4x_4.
\end{cases}
\end{equation}

Note that with the variables in Equation \eqref{eq0}, we have
$$
N_{m,a,b}:=\frac{m-2}{2}(a-b)+b=P_{m,\aaa}(x_1,x_2,x_3,x_4).
$$
In the case when $m=5$ and for each $\aaa\in\mathcal{A}$, we may suitably choose $s$ so that for any sufficiently large positive integer $N$, there is a pair $(a,b)$ of positive integers satisfying both $N=N_{5,a,b}$ and all the above three conditions. It will be done with the aid of Lemmas \ref{complete-residue} and \ref{leminterval}.
Therefore, we may conclude that $N\pra P_{5,\aaa}$. 

Returning back to the generalization of Cauchy's lemma, note that 
finding an (not necessarily non-negative) integer solution  of Equation \eqref{eq0} is equivalent to expressing the binary quadratic form $A_aX^2+2bXY+aY^2$ in variables $X$ and $Y$ as
$$
A_aX^2+2bXY+aY^2= \sum_{i=1}^4 a_i(X+x_iY)^2.
$$
Thus, this is a problem of finding a integral representation of special form of  the binary quadratic form by the diagonal quaternary quadratic form $Q_\aaa=\sum_{i=1}^4 a_iX_i^2$.

With this perspective, the second condition is exactly the local representability of quadratic forms. Hence the first and the second conditions may be considered as necessary conditions for Equation \eqref{eq0} to have an integer solution (see Lemma \ref{posrep} \ref{posrep:1}). Under those necessary conditions, the congruence condition in the third condition is a sufficient condition for Equation \eqref{eq0} to have an integer solution. 

Finally, if the inequality in the fourth condition holds, then any solution of Equation \eqref{eq0} must be non-negative (see the proof of Lemma \ref{posrep} \ref{posrep:2}).

\begin{rmk}{\color{white} a}
	\begin{enumerate}[label={\rm(\arabic*)},leftmargin=*]
		
		\item In \cite{Di2}, Dickson treated generalizations of Cauchy's lemma of the above kind for several $\aaa=(a_1,a_2,a_3,a_4)$, each of whose corresponding quadratic form $Q_\aaa=\sum_{i=1}^4 a_iX_i^2$ has {\em (genus) class number one}.
		We note that our method may also deal with the case that the class number of $Q_\aaa$ is greater than one.
				
		\item We note that this strategy may also be applied to the other cases when $m\neq 5$.
		Also, the set $\mathfrak{S}_s$ can systematically be computed once a positive integer $s$ is given.
		We refer the readers to \cite[Lemma 4.2]{K} for the proofs of the above generalization of Cauchy's lemma for several $\aaa$'s in detail.

		\item Questions on universality of the sum $P_{m,\aaa}$ over $\z$, that is, the solvability of Equation \eqref{eqPrepN} over $\bm{x}\in \z^k$ for any $N\in\nn$, have recently been studied by a number of authors, as Guy \cite{G} considered the numbers $P_m(x)$ with more general inputs $x\in\z$, which are called {\em generalized $m$-gonal numbers}.
		For example, Kane and Liu \cite{KL} proved that there exist a minimal positive integer $\gamma_m$ such that for any $\aaa\in\n^k$, the sum $P_{m,\aaa}$ is universal over $\z$ if and only if it represents every integer $N \le \gamma_m$.
		Using the arithmetic theory of quadratic forms, the constant $\gamma_m$ is determined for several small $m$; $\gamma_3=\gamma_6=8$ by Bosma and Kane \cite{BK}, $\gamma_4=15$ by $15$-Theorem, $\gamma_5=109$ by the first author \cite{J}, and $\gamma_8=60$ by Oh and the first author \cite{JO}.
		
		\item Note that for $m=3$ or $4$, the universality of $P_{m,\aaa}$ over $\z$ is equivalent to that over $\nn$. However, they become different for $m\ge 5$.
		Main difficulty that arise when concerning the problem over $\nn$ in comparison with that over $\z$ may be explained as follows.
		Classifying the sums $P_{5,\aaa}$ that are universal over $\z$ is completed by \cite{GS}, \cite{G}, \cite{J}, \cite{O}, and \cite{S}. 
		The most crucial part in proving all those results was to determine the set of integers that are represented by $P_{5,\aaa}$ over $\z$ for $\aaa=(a_1,a_2,a_3)\in\n^3$, and it was done based on the following observation;
		note that for an integer $N$, $P_{5,\aaa}(x_1,x_2,x_3)=N$ if and only if 
		\begin{equation}\label{eq2}
		24N+a_1+a_2+a_3=a_1(6x_1-1)^2+a_2(6x_2-1)^2+a_3(6x_3-1)^2.
		\end{equation}
		If an integer $X$ is coprime to $6$, then either $X$ or $-X$ is congruent to $-1$ modulo $6$. Therefore, Equation \eqref{eq2} has a solution $x_i\in\z$ if and only if 
		$$
		24N+a_1+a_2+a_3=a_1X_1^2+a_2X_2^2+a_3X_3^2
		$$ 
		has a solution $X_i\in\z$ with $\gcd(X_1X_2X_3,6)=1$.
		However, this method is no longer valid if we want to find a solution $x_i\in\nn$.
		It seems to be very difficult to determine which integers are represented by $P_{5,(a_1,a_2,a_3)}$ over $\nn$.
%
		As far as the authors' knowledge, Theorem \ref{63thm} is the first theorem of universality criterion for sums of (not generalized) polygonal numbers of order $m$ with $m\ge 5$.
	\end{enumerate}
\end{rmk}

The rest of the article is organized as follows.
In Section \ref{preliminaries}, we introduce several definitions, notations and well-known results on quadratic forms in terms of geometric language of $\z$-lattices. Section \ref{sec-lem} is devoted to introducing several lemmas that are needed in the proof of Theorem \ref{thmquatpolyrep}.
In Section \ref{quaternarysum}, we give a proof of Theorem \ref{thmquatpolyrep}.
Finally in Section \ref{sec-escalation}, we give a proof of Theorem \ref{63thm}, and we also classify all universal sums of pentagonal numbers by classifying all ``proper" ones (see Section \ref{sec-escalation} for the definition), which are listed in Table \ref{table3}.

\section{Preliminaries and general tools}\label{preliminaries}

In this section, we introduce several definitions, notations and well-known results on quadratic forms in the better adapted geometric language of quadratic spaces and lattices.
A $\z$-lattice $L=\z v_1+\z v_2+\dots+\z v_k$ of rank $k$ is a free $\z$-module equipped with non-degenerate symmetric bilinear form $B$ such that $B(v_i,v_j) \in \q$ for any $i,j$ with $1\le i, j \le k$. 
The $k\times k$ matrix $(B(v_i,v_j))$ is called the corresponding symmetric matrix to $L$, and we write
$$
L = (B(v_i,v_j)).
$$
The corresponding quadratic map is defined by $Q(v)=B(v,v)$ for any $v \in L$.
If $(B(v_i,v_j))$ is diagonal, then we simply write $L=\langle Q(v_1),\ldots,Q(v_k) \rangle$.

We say a $\z$-lattice $L$ is {\it positive definite} if $Q(v)>0$ for any non-zero vector $v \in L$, and we say $L$ is {\it integral} if $B(v,w) \in \z$ for any $v,w \in L$.
Throughout this article, we always assume that a $\z$-lattice is positive definite and integral.

For two $\z$-lattices $\ell$ and $L$, we say $\ell$ is {\em represented} by $L$ if there is a linear map $\sigma : \ell \to L$ such that 
$$
B(\sigma(x),\sigma(y))=B(x,y) \quad \text{for any $x,y \in \ell$}. 
$$
In this case, we write $\ell \ra L$, and such a linear map $\sigma$ is called a {\em representation} from $\ell$ to $L$.
We say $\ell$ and $L$ are {\em isometric} to each other if $\ell\ra L$ and $L\ra \ell$, and we write $\ell\cong L$.
For any prime $p$, we define localization of $L$ at $p$ by $L_p=L\otimes_\z \z_p$. 
We say $\ell$ is {\em locally represented} by $L$ if there is a local representation $\sigma_p : \ell_p \ra L_p$ which preserves the bilinear forms for any prime $p$. 
For a $\z$-lattice $L$, we define the {\em genus} $\gen(L)$ of $L$ as
$$
\gen(L)=\{K \text{ on } \q L \mid K_p \cong L_p \text{ for any prime } p\},
$$
where $\q L= \{ \alpha v \mid \alpha \in \q , v \in L\}$ is the quadratic space on which $L$ lies. The isometric relation induces an equivalence relation on $\gen(L)$, and we call the number of different equivalence classes in $\gen(L)$, the {\em  class number} of $L$.

Any unexplained notations and terminologies can be found in \cite{OM2}.   

The following is well-known local-global principle for $\z$-lattices.

\begin{thm}\label{localglobal}
	Let $\ell$ and $L$ be $\z$-lattices. If $\ell$ is locally represented by $L$, then $\ell\ra L'$ for some $L'\in\gen(L)$.
	Moreover, if the class number of $L$ is one, then $\ell \ra L$  if and only if $\ell$ is locally represented by $L$.
\end{thm}
\begin{proof}
	See 102:5 of \cite{OM2}.
\end{proof}
Due to the above theorem, we also say that $\ell$ is {\em locally represented} by $L$ if $\ell\ra L'$ for some $\z$-lattice $L'\in\gen(L)$, and we write $\ell \ra \gen(L)$.

For a binary $\z$-lattice $\ell=\z v_1 + \z v_2$, we simply write $\ell=[Q(v_1),B(v_1,v_2),Q(v_2)]$.
For an integer $A$, the set of all vectors $v\in L$ such that $Q(v)=A$ will be denoted by $R(A,L)$.

Let $A,a\in \n$ and $b\in\z$ be integers such that $Aa-b^2>0$, and let $\ell=\z v_1 + \z v_2=[A,b,a]$. 
Let $L$ be a $\z$-lattice and assume that there is a representation $\sigma: \ell \ra L$. Since $Q(\sigma(v_1))=Q(v_1)=A$, we necessarily have $\sigma(v_1)\in R(A,L)$.
Now, let $w$ be any vector in $R(A,L)$.
Our aim of the rest of this section is to develop a method to find a representation 
$$
\sigma : \ell \ra L \quad \text{such that} \quad \sigma(v_1)=w
$$
provided that $\ell \ra \gen(L)$. 
We should clear the following two general hurdles.
\begin{newenum}
	\item If the class number of $L$ is greater than $1$, 
	then the condition $\ell\ra \gen(L)$ does not guarantee $\ell\ra L$.
	\item Even though we know that $\ell$ is representable by $L$, 
	we cannot guarantee that there exists $\sigma:\ell\ra L$ with $\sigma(v_1)=w$ unless $|R(A,L)/O(L)|=1$.
\end{newenum}
Here, $R(A,L)/O(L)$ is the set of orbits induced by the action of $O(L)$ on $R(A,L)$ defined by $\sigma\cdot v=\sigma(v)$ for any $\sigma \in O(L)$ and $v\in R(A,L)$. Lemmas \ref{tool1} and \ref{tool2} which are stated below will explain methods to overcome the above two matters, respectively.

Let $K$ and $L$ be $\z$-lattices on the quadratic space $V$, and
let $A$ and $s$ be positive integers.
For $v\in R(A,K)$ and $(\alpha,\beta)\in (\z/s\z)^2$, we define
$$
R_{v}(K,\alpha,\beta,s):=\{u \in K/sK \mid  Q(u)\equiv \alpha \Mod{s}, \ B(u,v)\equiv \beta \Mod{s}\}
$$
and
$$
R_{v}(K,L,s):=\{\tau \in O(V) \mid \tau(sK)\subset L, \ \tau(v)\in L\}.
$$
Note that the above sets depend on the positive integer $A$ although it is not indicated.
A coset (or, a vector in the coset) $u\in R_v(K,\alpha,\beta,s)$ is said to be {\em good} (with respect to $K,L,(\alpha,\beta),s,$ and $v$) if there exist a $\tau \in R_v(K,L,s)$ such that 
$$
\text{
$\tau(\tilde{u})\in L$ for any $\tilde{u}\in K$ with $\tilde{u}\equiv u \Mod{sK}$.}
$$
Indeed, the above condition holds if $\tau(\tilde{u})\in L$ is satisfied for a vector $\tilde{u}$ in the coset $u$.
The set of all good cosets (or, all good vectors) in $R_v(K,\alpha,\beta,s)$ is denoted by $R_v^L(K,\alpha,\beta,s)$. If $R_v(K,\alpha,\beta,s)=R_v^L(K,\alpha,\beta,s)$ for any $v\in R(A,K)$, we write
$$
K \prec_{A,(\alpha,\beta),s} L.
$$
The following lemma is the first tool of our method.
\begin{lem}\label{tool1}
Let $A,s$ be positive integers, $(\alpha,\beta)\in (\z/s\z)^2$, and define
$$
S_{A,\alpha,\beta,s}:=\{[A,b,a] \mid (a,b)\equiv (\alpha,\beta) \Mod{s}, \ Aa-b^2>0\}.
$$
Let $K,L$ be two $\z$-lattices on the same quadratic space such that $K \prec_{A,(\alpha,\beta),s} L$, and let $\ell$ be a binary $\z$-lattice in $S_{A,\alpha,\beta,s}$.
Then $\ell \ra K$ implies $\ell\ra L$.
In particular, if $M \prec_{A,(\alpha,\beta),s} L$ for any $M\in\gen(L)$, then 
$\ell \ra \gen(L)$ implies $\ell\ra L$.
\end{lem}
\begin{proof}
Let $\ell=\z v_1 + \z v_2 = [A,b,a]$ and let $\sigma:\ell \ra K$ be a representation. Put $v:=\sigma(v_1)$ and $u:=\sigma(v_2)$. Note that $v,u\in K$ and
$$
Q(v)=A,\ Q(u)=a\equiv \alpha \Mod{s},\ \text{and} \ B(u,v)=b\equiv \beta\Mod{s}.
$$
Therefore, $v\in R(A,K)$ and $u \in R_v(K,\alpha,\beta,s)$. Since  $K \prec_{A,(\alpha,\beta),s} L$, there exists $\tau\in R_v(K,L,s)$ such that $\tau(u)\in L$. Note that $\tau(v)\in L$ by the definition of the set $R_v(K,L,s)$. Therefore, $\tau\circ \sigma$ is a representation from $\ell$ into $L$. This completes the proof.
\end{proof}

To introduce the second tool, we need to define something more, which is quite similar to the above. 
Let $L$ be a $\z$-lattice on a quadratic space $V$, $s$ be a positive integer, and $(\alpha,\beta)\in(\z/s\z)^2$. For $v,w\in R(A,L)$, we define
$$
R_{v,w}(L,s):=\{\tau \in R_{v}(L,L,s) \mid \tau(v)=w\},
$$
A coset (or, a vector in the coset) $u\in R_{v}(L,\alpha,\beta,s)$ is said to be {\em good} with respect to $L,(\alpha,\beta),s,v,$ and $w$ if there exists a $\tau \in R_{v,w}(L,s)$ such that 
$$
\text{
	$\tau(\tilde{u})\in L$ for any $\tilde{u}\in L$ with $\tilde{u}\equiv u \Mod{sL}$.}
$$
The set of all good vectors in $R_v(L,\alpha,\beta,s)$ is denoted by $R_v^w(L,\alpha,\beta,s)$. We write
$$
L \prec_{A,(\alpha,\beta),s} w
$$
if $R_v(L,\alpha,\beta,s)=R_v^w(L,\alpha,\beta,s)$ for any $v\in R(A,L)$. 

\begin{lem}\label{tool2}
Let $A,s$ be positive integers, $(\alpha,\beta)\in (\z/s\z)^2$, 
and let 
$$
\ell=\z v_1+\z v_2=[A,b,a]\in S_{A,\alpha,\beta,s},
$$  
where the set $S_{A,\alpha,\beta,s}$ is defined in Lemma \ref{tool1}.
Let $L$ be a $\z$-lattice and let $w\in R(A,L)$ be a vector such that $L \prec_{A,(\alpha,\beta),s} w$.
If $\ell \ra L$, then there exists a representation $\sigma :\ell\ra L$ such that $\sigma(v_1)=w$.
\end{lem}
\begin{proof}
The proof is quite similar to that of Lemma \ref{tool1}, so it is left to the reader.
\end{proof}

\begin{cor}\label{tool3}
For a $\z$-lattice $L$, positive integers $A,s$, and $w\in R(A,L)$, define
$$
\mathfrak{S}_{w,s}(L):=\{ (\alpha,\beta)\in (\z/s\z)^2 \mid L \prec_{A,(\alpha,\beta),s}w\}.
$$
Moreover, for any $\z$-lattice $K\in\gen(L)$ such that $K\not\cong L$, define 
$$
\mathfrak{S}_{L,s}(K):=
\{ (\alpha,\beta)\in (\z/s\z)^2 \mid  K\prec_{A,(\alpha,\beta),s} L\}.
$$	
Put 
$$
\mathfrak{S}_s=\mathfrak{S}_{L,w,s}:=\mathfrak{S}_{w,s}(L) \cap \left(\bigcap_{K \in \gen(L)\setminus[L]} \mathfrak{S}_{L,s}(K)\right).
$$
Let $(\alpha,\beta)\in\mathfrak{S}_s$ and let $\ell=\z v_1 + \z v_2=[A,b,a]$ be a binary $\z$-lattice in the set $S_{A,\alpha,\beta,s}$ defined in Lemma \ref{tool1}.
Then there exists a representation $\sigma:\ell\ra L$ such that $\sigma(v_1)=w$, provided that $\ell\ra \gen(L)$.
\end{cor}
\begin{proof}
	This is a direct consequence of Lemma \ref{tool1} and \ref{tool2}.
\end{proof}

\begin{rmk}\label{rmk-mathfrakS}
%
{\color{white}a}
\begin{enumerate}[label={\rm(\arabic*)},leftmargin=*]
	\item One may easily show that for any $\sigma\in O(K)$, the map 
	$$
	R_v(K,L,s) \ra R_{\sigma(v)}(K,L,s) \  : \  \tau \mapsto \tau\circ\sigma^{-1}
	$$
	is a bijection. Consequently, for any $\sigma\in O(K)$, we have
	$$
	R_{v}(K,\alpha,\beta,s)=R_{v}^L(K,\alpha,\beta,s) \ \Leftrightarrow \ R_{\sigma(v)}(K,\alpha,\beta,s)=R_{\sigma(v)}^L(K,\alpha,\beta,s).
	$$
	Thus, when one determine the set $\mathfrak{S}_{L,s}(K)$, it suffices to check whether or not $R_{v}(K,\alpha,\beta,s)=R_{v}^L(K,\alpha,\beta,s)$ only for representative vectors of $R(A,K)/O(K)$.
	The same is also true for determining the set $\mathfrak{S}_{w,s}(L)$.
	
	\item\label{rmk-mathfrakS:2} Let $s_1,s_2\in\n$ be integers such that $s_1 \mid s_2$, and let $(\alpha_1,\beta_1)\in (\z/s_1\z)^2$ and $(\alpha_2,\beta_2)\in (\z/s_2\z)^2$.
	Assume that $(\alpha_2,\beta_2) \Mod{s_1} = (\alpha_1,\beta_1)$ in $(\z/s_1\z)^2$. Then one may show from the definition that
	$$
	K \prec_{A,(\alpha_1,\beta_1),s_1} L \ \Rightarrow \ 
	K \prec_{A,(\alpha_2,\beta_2),s_2} L.
	$$
	Therefore, we may conclude that
	$$
	\{(\alpha,\beta)\in(\z/s_2\z)^2 \mid (\alpha,\beta)\Mod{s_1}\in \mathfrak{S}_{L,s_1}(K)\} \subseteq \mathfrak{S}_{L,s_2}(K).
	$$
	The same is also true for the sets $\mathfrak{S}_{w,s_1}(L)$ and $\mathfrak{S}_{w,s_2}(L)$.
\end{enumerate}
\end{rmk}

In this article, we use Lemmas \ref{tool1}, \ref{tool2}, and \ref{tool3} only for quaternary $\z$-lattices $K$ and $L$.
A MAGMA-based computer program for computing the sets defined the above for quaternary $\z$-lattices is available upon request to the authors.

\section{Lemmas}\label{sec-lem}
%
%
%
In this section, we introduce several lemmas which will be used to prove Theorem \ref{thmquatpolyrep}.
Throughout the rest of the article, let us set several notations.
For each $\aaa=(a_1,a_2,a_3,a_4)\in \mathbb{N}^4$, we put $A=A_\aaa=\sum_{i=1}^4a_i$, and we define the quaternary diagonal $\z$-lattice $L_\aaa$ with basis $\{w_1,w_2,w_3,w_4\}$ by
$$
L_\aaa=\z w_1+\z w_2 + \z w_3 + \z w_4 =\df{a_1,a_2,a_3,a_4}.
$$
For convenience, the vector $w_1+w_2+w_3+w_4$ of $L_\aaa$ will be denoted by $w_\aaa$.
We always assume that $a_1\le a_2\le a_3 \le a_4$.

\begin{lem}\label{posrep}
	Let $\aaa=(a_1,a_2,a_3,a_4)\in\mathbb{N}^4$ with $a_1\le a_2\le a_3 \le a_4$, $a\in \n$, and $b \in \z$. Define 
	$$
	N=N_{m,a,b}:=\frac{m-2}{2}(a-b)+b.
	$$
	Assume that the following system of diophantine equations
	\begin{equation}\label{eq1}
	\begin{cases}
	a=a_1x_1^2+a_2x_2^2+a_3x_3^2+a_4x_4^2\\
	b=a_1x_1+a_2x_2+a_3x_3+a_4x_4
	\end{cases}	
	\end{equation}
	has an integer solution $x_1,x_2,x_3,x_4\in\z$. Then we have
	\begin{newenum}
		\item\label{posrep:1} $a\equiv b\Mod{2}$, $Aa-b^2\ge 0$, and $N = P_{m,\aaa}(x_1,x_2,x_3,x_4)$,
		\item\label{posrep:2} if $b\ge \sqrt{A-a_1}\cdot\sqrt{a} =\sqrt{a_2+a_3+a_4}\cdot\sqrt{a}$, then $N \pra P_{m,\aaa}$.
	\end{newenum}
\end{lem}
\begin{proof} 
	Since $x_i^2\equiv x_i\Mod{2}$, we have $a\equiv b\Mod{2}$, and the inequality $Aa-b^2\ge0$ is nothing but Cauchy-Schwarz inequality. 
	Moreover,
	$$
	N=\frac{m-2}{2}(a-b)+b
	=\sum_{i=1}^4 a_i\left(\frac{m-2}{2}(x_i^2-x_i)+x_i\right)
	=P_{m,\aaa}(x_1,x_2,x_3,x_4).
	$$
	To prove \ref{posrep:2}, it suffices to show that if $b\ge \sqrt{A-a_1}\cdot\sqrt{a}$, then $x_1,x_2,x_3,x_4\ge 0$.
	Letting $X_h=\sqrt{a_h}x_h/\sqrt{a}$ for any $1\le h \le 4$, we have
	$$
	\begin{cases}
	X_1^2+X_2^2+X_3^2+X_4^2=1\\
	\sqrt{a_1}X_1+\sqrt{a_2}X_2+\sqrt{a_3}X_3+\sqrt{a_4}X_4=b/\sqrt{a}.
	\end{cases}
	$$
	Assume to the contrary that $X_h<0$ for some $1\le h \le 4$, and let $i<j<k$ be indices such that $\{h,i,j,k\}=\{1,2,3,4\}$.
	Then $(X_i,X_j,X_k)\in \mathbb{R}^3$ satisfies
	\begin{align}
	X_i^2+X_j^2+X_k^2&=1-X_h^2<1, \quad \text{ and } \label{eq:sphere}\\ 
	\sqrt{a_i}X_i+\sqrt{a_j}X_j+\sqrt{a_k}X_k&=b/\sqrt{a}-\sqrt{a_i}X_i>b/\sqrt{a} \label{eq:halfspace}\\ 
	&\ge \sqrt{a_2+a_3+a_4}\ge \sqrt{a_i+a_j+a_k} . \nonumber
	\end{align}
	However, this is a contradiction since two regions for $(X_i,X_j,X_k)$ in $\mathbb{R}^3$ defined by the inequalities \eqref{eq:sphere} and \eqref{eq:halfspace} respectively have no intersection point.
	This proves the lemma.
\end{proof}

The above lemma leads us to study the following general question: for which $a\in \n$ and $b\in\z$, there exists an integer solution $(x_1,x_2,x_3,x_4)\in \z^4$ of Equation \eqref{eq1}?
The following lemma connects a solvability of Equation \eqref{eq1} over $\z$ with an existence of a particular representation of the binary $\z$-lattice $[A,b,a]$ by a diagonal quaternary $\z$-lattice $L_\aaa$.

\begin{lem}\label{eq1-latticerep}
	Let $\aaa=(a_1,a_2,a_3,a_4)\in \n^4$, and let $a$ and $b$ be integers such that 
	$$a\equiv b\Mod{2} \quad  \text{and} \quad Aa-b^2>0.$$
	Then the following are equivalent. 
	\begin{newenum}
		\item Equation \eqref{eq1} has an integer solution $(x_1,x_2,x_3,x_4)\in\z^4$.
		\item There exists a representation $\sigma : [A,b,a]\ra L_\aaa$ such that $\sigma(v_1)=w_\aaa$.
	\end{newenum}
\end{lem}
\begin{proof}
Assume that Equation \eqref{eq1} has an integer solution $(x_1,x_2,x_3,x_4)\in\z^4$.
Let us define a linear map $\sigma$ from $[A,b,a]$ to $L_\aaa$ by
$$
\sigma(v_1) = w_\aaa= w_1+w_2+w_3+w_4 \quad \text{and} \quad
\sigma(v_2) = \sum_{i=1}^4 x_iw_i.
$$
Then one may have
$$
\begin{cases}
Q(\sigma(v_1))=\sum_{i=1}^4 a_i = A = Q(v_1),\\
Q(\sigma(v_2))=\sum_{i=1}^4 a_ix_i^2 = a = Q(v_2),\\
B(\sigma(v_1),\sigma(v_2))= \sum_{i=1}^4 a_ix_i = b = B(v_1,v_2).
\end{cases}
$$
Therefore, $\sigma : [A,b,a] \ra L_\aaa$ is a representation that we want to find.

Conversely, assume that there exists a representation $\sigma : [A,b,a] \ra L_\aaa$ such that $\sigma(v_1) = w_\aaa$. 
Then $\sigma(v_2)=\sum_{i=1}^4 x_iw_i$ for some integers $x_i\in\z$.
One may easily show that this $(x_1,x_2,x_3,x_4)\in\z^4$ is a solution of Equation \eqref{eq1}.
\end{proof}

Now, we introduce a method to find a representation $\sigma :[A,b,a] \ra L_\aaa$ with $\sigma(v_1)=w_\aaa$.

\begin{lem}\label{rep-gen-to-lattice}
Let $s$ be a positive integer. For any $\z$-lattice $K\in\gen(L_\aaa)$, define 
$$
\mathfrak{S}_{\aaa,s}(K):=
\begin{cases}
\mathfrak{S}_{w_\aaa,s}(L_\aaa) & \text{if } K\cong L_\aaa,\\[5pt]
\mathfrak{S}_{L_\aaa,s}(K)& \text{otherwise},
\end{cases}
$$
where the sets $\mathfrak{S}_{w_\aaa,s}(L_\aaa),\mathfrak{S}_{L_\aaa,s}(K)\subseteq (\z/s\z)^2$ are defined in Lemma \ref{tool3}, and let 
$$
\mathfrak{S}_s=\mathfrak{S}_{\aaa,s}:=\bigcap_{K \in \gen(L_\aaa)} \mathfrak{S}_{\aaa,s}(K).
$$ 
Let $(\alpha,\beta)\in\mathfrak{S}_s$ and let $\ell=\z v_1 + \z v_2=[A,b,a]\in S_{A,\alpha,\beta,s}$ be a binary $\z$-lattice such that 
$$
a\equiv b \Mod{2}\quad \text{and} \quad [A,b,a]\ra \gen(L_\aaa).
$$
Then there exists a representation $\sigma:[A,b,a]\ra L_\aaa$ such that $\sigma(v_1)=w_\aaa$.
\end{lem}
\begin{proof}
This is a particular case of Corollary \ref{tool3}.
\end{proof}

We introduce two technical lemmas needed for the proof of Theorem \ref{thmquatpolyrep}.
%

\begin{lem}\label{complete-residue}
	Let $s$ be a positive even integer.
	For any subset $\mathfrak{S}$ of $(\z/s\z)^2$, define
	$$
	\mathcal{N}(\mathfrak{S}):=
	\left\{ \frac{3}{2}(\alpha-\beta)+\beta \Mod{s/2} \mid (\alpha,\beta)\in \mathfrak{S}, \ \alpha\equiv\beta \Mod{2} \right\}.
	$$
	Assume that $\mathcal{N}(\mathfrak{S})$ is the complete set of residue modulo $s/2$.
	Then for any $N\in\n$, there exists a residue $r_{N,s}$ modulo $3s$ satisfying the following:
	\begin{gather}\!\!\!\!\!\!\!\!\!\!\!\!\!
		\text{for any $b\in\z$ with $b\equiv r_{N,s} \Mod{3s}$ and $a=\frac{2}{3}(N-b)+b$, we have}\label{condr}\\
		N\equiv b\Mod{3}, \quad a\equiv b \Mod{2}, \quad \text{and}\quad (a,b)\Mod{s} \in \mathfrak{S}, \nonumber
	\end{gather}
\end{lem}
\begin{proof}
	Let $N$ be a positive integer. 
	From the assumption, there exists an $(\alpha,\beta)\in \mathfrak{S}$ such that $N \equiv \frac{3}{2}(\alpha-\beta)+\beta \Mod{s/2}$.
	Note that $\beta \equiv 3\alpha - 2N \Mod{s}$.
	Let us set $r_{N,s}=r_{N,\alpha,\beta,s}=3\alpha-2N$. 
	Then for any $b\equiv r_{N,s} \Mod{3s}$,
	$$
	b \equiv \beta \Mod{s}, \ N\equiv b \Mod{3}, \ a\in\n, \ a\equiv \alpha \Mod{s}, \ \text{and} \ a\equiv b\Mod{2}.
	$$
	In particular, $(a,b) \Mod{s} = (\alpha,\beta)\in \mathfrak{S}$. This proves the lemma.
\end{proof}

\begin{lem}\label{leminterval}
For $N,A\in \n$, define the interval $I_{N,A}=[L_{N,A},U_{N,A})$ by
$$
\left[
\sqrt{\frac{2(A-1)}{3}N+\left(\frac{A-1}{6}\right)^2 } +\frac{A-1}{6}, \
\sqrt{\frac{2A}{3}N+\left(\frac{A}{6}\right)^2} +\frac{A}{6} \right)
.
$$
For any $b\in I_{N,A}$, $Aa-b^2>0$ and $b\ge \sqrt{A-1}\cdot\sqrt{a}$, where $a=\frac{2}{3}(N-b)+b$.
\end{lem}

\begin{proof}
By putting $a=\frac{2}{3}(N-b)+b$ and completing the square, we have 
$$
\frac{b^2}{A} < a \le \frac{b^2}{A-1} \  \Leftrightarrow \ L_{N,A}\le b < U_{L,A}.
$$
This proofs the lemma.
\end{proof}

\begin{rmk}\label{rmkinverval}
Note that $U_{N,A}-L_{N,A}$ (which is the length of $I_{N,A}$) and $L_{N,A}$ are  increasing functions of $N$ and they diverges to infinity as $N\ra \infty$, respectively.
\end{rmk}

\section{Proof of Theorem \ref{thmquatpolyrep}}\label{quaternarysum}

In this section, we prove Theorem \ref{thmquatpolyrep}.

\begin{proof}[Proof of Theorem \ref{thmquatpolyrep}]
The proofs for all $\aaa\in\mathcal{A}$ are quite similar to each other except for the proofs for $\aaa=(1,2,3,3)$, $(1,2,4,5)$ and $(1,2,4,7)$. 
Therefore, we describe general strategy of the proof and provide the proofs for several $\aaa\in \mathcal{A}$, including the cases when $\aaa=(1,2,3,3)$, $(1,2,4,5)$ and $(1,2,4,7)$.

For a positive integer $s$, let $\mathfrak{S}_s=\mathfrak{S}_{\aaa,s}$ be the subset of $(\z/s\z)^2$ defined in Lemma \ref{rep-gen-to-lattice}.
If $\mathfrak{S}_1=\{(0,0)\}$, then we will set $s_\aaa=1$. Otherwise, we will find a positive even integer $s=s_\aaa$, as small as possible, for which the set $\mathcal{N}(\mathfrak{S}_s)$ defined in Lemma \ref{complete-residue} is the complete set of residue modulo $s/2$.

Let $N\ge N_\aaa$ be a positive integer. With our choice of $s=s_\aaa$, let us take a residue $r_{N,s}$ modulo $3s$ satisfying \eqref{condr} in Lemma \ref{complete-residue} (if $s=1$, then take $r_{N,s}=N$).
Depending on the set $\mathfrak{S}_s$, there may be many choices for $r_{N,s}$ (see the proof of Lemma \ref{complete-residue}).
In some cases, we don't have to consider the choice of $r_{N,s}$, but in some other cases, we should specify $r_{N,s}$ according as the residue of $N$ modulo $s/2$.

Let $b$ be any positive integer such that $b\equiv r_{N,s} \Mod{3s}$. 
Then by Lemma \ref{complete-residue},
$$
\ N\equiv b \Mod{3}, \quad a\equiv b\Mod{2}, \quad \text{and} \quad (a,b)\Mod{s} \in \mathfrak{S}_s,
$$
where $a=\frac{2}{3}(N-b)+b \in \n$. Assume that 
\begin{gather} 
Aa-b^2>0, \quad b\ge \sqrt{A-1}\cdot\sqrt{a}, \quad \text{and} \label{condb:1}\\
\quad [A,b,a] \ra \gen(L_\aaa). \label{condb:2}
\end{gather}
Then, since $(a,b) \Mod{s} \in \mathfrak{S}_s$, Lemma \ref{rep-gen-to-lattice} implies that 
there exists a representation $\sigma : [A,b,a]\ra L_\aaa$ with $\sigma(v_1)=w_\aaa$.
Thus by Lemma \ref{eq1-latticerep}, Equation \eqref{eq1} has an integer solution.
Finally, by Lemma \ref{posrep} \ref{posrep:2} (note that $a_1=1$ for any $\aaa\in\mathcal{A}$), we have
$N=\frac{3}{2}(a-b)+b \pra P_{5,\aaa}$, which proves the statement of the theorem.

Now our task is to take $b\in\n$ with $b\equiv r_{N,s} \Mod{3s}$ so well that they satisfy \eqref{condb:1} and \eqref{condb:2}.
For the rest of the proof, $a'\in\n$ and $b'\in\z$ will denote integers satisfying 
$$
a'\equiv b' \Mod{2} \quad \text{and} \quad Aa'-b'^2>0.
$$
Note that the necessary and sufficient condition for $[A,b',a']\ra \gen(L_\aaa)$ can be verified by Theorems $1$ and $3$ of \cite{OM1}, in terms of congruence conditions on $Aa'-b'^2$, $a'$, and $b'$.
From this, we will find a positive integer $B=B_\aaa$ satisfying the following:
\begin{enumerate}[leftmargin=*]
\item[] for any positive integer $b_0$ such that $b_0\equiv r_{N,s} \Mod{3s}$, there exists an integer $b\in \{b_0+3sk \mid 0\le k \le B-1\}$ satisfying \eqref{condb:2}.
\end{enumerate}
Let $I=I_{N,A}=[L_{N,A},U_{N,A})$ be the interval defined in Lemma \ref{leminterval}. 
Note that $N_\aaa$ is taken as the smallest positive integer satisfying the following inequality:
\begin{equation}\label{condI}
U_{N,A}-L_{N,A}\ge 3sB
\end{equation}
(see Remark \ref{rmkinverval}).
Since $I$ contains $3sB$ consecutive integers and because of our choice of $B$, there is an integer $b\in I$ satisfying \eqref{condb:2}.
Moreover, \eqref{condb:1} also holds by Lemma \ref{leminterval}.
This will complete the proof.\\

The proofs for any $\aaa\in\mathcal{A}$ follow the same stream once we take positive integers 
$$s_\aaa \quad  \text{and} \quad  B_\aaa$$
appropriately to satisfy the conditions described above.
We categorize $\aaa\in\mathcal{A}$ into the following four types (see Tables \ref{table-data1} and \ref{table-data2}):
$$
\begin{array}{ll}
\textbf{Type 1} & \text{the class number of $L_\aaa=1$ and $R(A,L)/O(L)=\{[w_\aaa]\}$.} \\
\textbf{Type 2} & \text{the class number of $L_\aaa=1$ and $|R(A,L)/O(L)|>1$.} \\
\textbf{Type 3} & \text{the class number of $L_\aaa>1$ and $\aaa \neq (1,2,3,3), \ (1,2,4,5), \ (1,2,4,7)$.} \\
\textbf{Type 4} & L_\aaa=\df{1,2,3,3}, \ \df{1,2,4,5}, \text{ or } \df{1,2,4,7},
\end{array}
$$
and provide the proofs for at least one $\aaa$ of each types in detail, namely,
for the case when $\aaa$ is one of 
$(1,1,1,4)$, $(1,2,2,3)$, $(1,1,2,5)$, $(1,1,3,6)$, $(1,2,3,8)$, and all of the three {\bf Type 4} vectors.
For all the other cases, we provide integers $s_\aaa$ and $B_\aaa$ in Tables \ref{table-data1} and \ref{table-data2}.\\


\noindent {\bf (Case 1)} $\aaa=(1,1,1,4)$, $A=A_\aaa=7$, $L=L_\aaa=\df{1,1,1,4}$.

Note that this $\aaa$ is of {\bf Type 1}, precisely, the class number of $L$ is $1$, and for any $w\in R(A,L)$, there is a $\tau\in O(L)$ such that $\tau(w)=w_\aaa$.
Therefore, for any $\aaa$ of {\bf Type 1}, we have  $\mathfrak{S}_{\aaa,1}=\mathfrak{S}_{\aaa,1}(L)=\{(0,0)\}$.
Let us set $s=1$.

Note that $[7,b',a'] \ra \gen(L)$ if and only if
$7a'-b'^2$ is not of the form $2^{2c}(8d+7)$	
for some non-negative integers $c$ and $d$.
Moreover, if $[7,b',a'] \ra \gen(L)$, then there is a representation $\sigma : [7,b',a'] \ra L$ with $\sigma(v_1)=w_\aaa$.

Let $N\ge N_\aaa=5453$ be a positive integer, and let us set $B=4$ and $r_{N,s}=N$.
Let $b_0\in\n$ be any integer such that $b_0 \equiv r_{N,s} \Mod{3}$.
For any integer $i$, define
\begin{equation}\label{defaibi}
b_i:=b_0+3i \quad \text{and}\quad
a_i:=\frac{2}{3}(N-b_i)+b_i=a_0+i.
\end{equation}
Note that $a_k\equiv 2\Mod{4}$ for some $k \in \{0,1,2,3\}$.
Let $b=b_k$ for such $k$ and let $a=\frac{2}{3}(N-b)+b$.
Note that $a=a_k$, and $7a-b^2\equiv 2 \Mod{4}$.
Hence, $[7,b,a]\ra \gen(L)$. Therefore, $b=b_k$ satisfies \eqref{condb:2}.

Note that the length of $I_{N,A}$ is greater than or equal to $3sB=12$ by \eqref{condI}.
Therefore, there is an integer $b\in I_{N,A}$ satisfying \eqref{condb:1} and \eqref{condb:2}. Hence $N \pra P_{5,\aaa}$.\\


\noindent {\bf (Case 2)} $\aaa=(1,2,2,3)$, $A=A_\aaa=8$, $L=L_\aaa=\df{1,2,2,3}$.

Note that this $\aaa$ is of {\bf Type 2}, precisely, the class number of $L$ is $1$, and for any $w\in R(A,L)$, there is a $\tau\in O(L)$ such that $\tau(w)$ is equal to one of the following three vectors of $L$:
$$
w_\aaa, \quad w'=2w_1+w_2+w_3, \quad \text{and} \quad w''=2w_2.
$$
In other word, $R(A,L)/O(L)=\left\{[w_\aaa],[w'],[w''] \right\}$.
Note that $R_{w}(L,\alpha,\beta,2)=R_{w}^{w_\aaa}(L,\alpha,\beta,2)$ if and only if $(\alpha,\beta)\in(\z/2\z)^2$ is in
$$
\begin{array}{ll}
\{(0,0),(0,1),(1,0),(1,1)\} & \text{if } w=w_\aaa,\\
\{(0,0),(0,1),(1,1)\} & \text{if } w=w',\\
\{(0,1),(1,1)\} & \text{if } w=w''.
\end{array}
$$
As a sample, consider the case when $w=w'$ and $(\alpha,\beta)=(0,0)$. Note that 
$$
u=\sum_{i=1}^4 y_iw_i \in R_{w'}(L,0,0,2) \quad \Leftrightarrow \quad  y_1\equiv y_4 \Mod{2}.
$$
Also, note that there is a $\sigma_T \in R_{w',w_\aaa}(L,2)$ defined by
$$
\sigma_T(w_j)=\sum_{i=1}^4 t_{ij}w_i \text{ for each } 1\le j \le 4, \text{ where } T=(t_{ij})=\frac{1}{2}\cdot { \tiny \setlength\arraycolsep{2pt} \begin{pmatrix}
1&0&0&3 \\ 0&0&2&0 \\ 0&2&0&0 \\ 1&0&0&-1
\end{pmatrix}}.
$$
Thus, $\sigma_T(u)\in L$ for any $u\in R_{w'}(L,0,0,2)$, hence $R_{w'}(L,0,0,2)=R_{w'}^{w_\aaa}(L,0,0,2)$.
Note that $\mathfrak{S}_{2}=\mathfrak{S}_{\aaa,2}(L)=\{(0,1),(1,1)\}$, and
$\mathcal{N}(\mathfrak{S}_{2})=\{0\}$.
Let us set $s=2$.

Note that $[8,b',a'] \ra \gen(L)$ if and only if 
$8a'-b'^2$ is not of the form $2^{2c}(8d+5)$	
for some non-negative integers $c$ and $d$.
Assume that $a'\equiv b'\equiv 1\Mod{2}$. Then $8a'-b'^2\equiv 7 \Mod{8}$, hence $[8,b',a'] \ra \gen(L)$. Moreover, since $(a',b')\Mod{2} \in \mathfrak{S}_{2}$, there is a representation $\sigma : [8,b',a'] \ra L$ with $\sigma(v_1)=w_\aaa$.

Let $N\ge N_\aaa=1529$ be a positive integer, and let us set $B=1$. 
Note that the length of $I_{N,A}$ is greater than or equal to $3sB=6$ by \eqref{condI}.
Let us choose the residue $r_{N,2}$ modulo $6$ corresponding to $(1,1)\in \mathfrak{S}_2$.
Let $b$ be a positive integer in $I$ such that $b \equiv r_{N,s} \Mod{6}$, and let $a=\frac{2}{3}(N-b)+b$.
Note that $a\equiv b\equiv 1 \Mod{2}$.
Hence, $[8,b,a]\ra \gen(L)$. Thus, this integer $b$ satisfies \eqref{condb:1} and \eqref{condb:2}.
Hence $N \pra P_{5,\aaa}$.\\


\noindent {\bf (Case 3)} $\aaa=(1,1,2,5)$, $A=A_\aaa=9$, $L=L_\aaa=\df{1,1,2,5}$.

Note that this $\aaa$ is of {\bf Type 3}, precisely, the class number of $L$ is $2$, and all isometric classes in the genus are
$$
L=\df{1,1,2,5} \quad \text{and} \quad 
K=\z u_1 + \z u_2 + \z u_3 + \z u_4 = \df{1}\perp\begin{pmatrix}2&1&1\\1&2&0\\1&0&4\end{pmatrix}.
$$
First, we may check that $\mathfrak{S}_{\aaa,2}(K)=\{(0,0),(1,1)\}$. 
In fact, 
$$
R(A,K)/O(K)=\{[v':=3u_1],\ [v'':=u_1+2u_2],\ [v''':=u_1+2u_2-u_4]\},
$$ 
and $R_v(K,\alpha,\beta,2)=R_v^L(K,\alpha,\beta,2)$ if and only if $(\alpha,\beta)\in(\z/2\z)^2$ is in
$$
\begin{array}{ll}
\{(0,0),(0,1),(1,0),(1,1)\} & \text{if } v=v'\text{ or } v'',\\
\{(0,0),(1,1)\} & \text{if } v=v'''.
\end{array}
$$
As a sample, let us show $R_v(K,\alpha,\beta,2)=R_v^L(K,\alpha,\beta,2)$ for $v=v''=u_1+2u_2$ and for any $(\alpha,\beta)\in(\z/2\z)^2$.
For a $4\times 4$ matrix $T=(t_{ij})$ with entries in $\q$, define a linear map $\sigma_T:\q K \ra \q L$ by $\sigma_T(u_j)=\sum_{i=1}^4 t_{ij}w_i$ for any $1\le j \le 4$. Then we may check that $\sigma_{T_1},\sigma_{T_2},\sigma_{T_3}\in R_v(K,L,2)$, where
$$
T_1=\frac{1}{2}\cdot { \tiny \setlength\arraycolsep{2pt} \begin{pmatrix}
	2&0&0&0 \\ 0&1&0&4 \\ 0&1&2&0 \\ 0&1&0&0
	\end{pmatrix}},\ \
T_2=\frac{1}{2}\cdot { \tiny \setlength\arraycolsep{2pt} \begin{pmatrix}
	2&0&0&0 \\ 0&0&1&3 \\ 0&2&1&1 \\ 0&0&1&-1
\end{pmatrix}},\ \ \text{and} \ \
T_3=\frac{1}{2}\cdot { \tiny \setlength\arraycolsep{2pt} \begin{pmatrix}
	2&0&0&0 \\ 0&1&1&-3 \\ 0&1&-1&1 \\ 0&1&1&1
	\end{pmatrix}}.
$$
In fact, we have $\sigma_{T_1}(v)=\sigma_{T_3}(v)=w_\aaa$ and $\sigma_{T_2}(v)=w_1+2w_3$.
Let $u=\sum_{i=1}^4 y_iu_i \in R_v(K,\alpha,\beta,2)$.
Then $\tau(u)\in L$ for some $\tau\in R(K,L,2)$, namely,
$$
\begin{array}{ll}
\sigma_{T_1}(u)\in L & \text{if } y_2\equiv 0\Mod{2},\\
\sigma_{T_2}(u)\in L & \text{if } y_2\equiv 1\Mod{2} \text{ and } y_3+y_4\equiv 0\Mod{2},\\
\sigma_{T_3}(u)\in L & \text{if } y_2\equiv 1\Mod{2} \text{ and } y_3+y_4\equiv 1\Mod{2}.
\end{array}
$$
This proves $R_v(K,\alpha,\beta,2)=R_v^L(K,\alpha,\beta,2)$ for $v=v''$.

On the other hand, note that $R(A,L)/O(L)=\left\{[w_\aaa],[w'],[w''],[w'''] \right\}$, where
$$
w'=3w_1, \quad w''=2w_1+w_4, \quad \text{and} \quad w'''=w_1+2w_3.
$$
By checking whether or not $R_{w}(L,\alpha,\beta,24)=R_{w}^{w_\aaa}(L,\alpha,\beta,24)$ for $w=w_\aaa,w',w'',w'''$ and for each $(\alpha,\beta)\in(\z/24\z)^2$,
we have
$$
\mathfrak{S}_{\aaa,24}(L) \subseteq \{(\alpha,\beta)\in (\z/24\z)^2 \mid \alpha \equiv \beta \Mod{2}\}\subseteq \mathfrak{S}_{\aaa,24}(K).
$$
The last set inclusion is due to Remark \ref{rmk-mathfrakS} \ref{rmk-mathfrakS:2} with $\mathfrak{S}_{\aaa,2}(K)=\{(0,0),(1,1)\}$.
Thus, we have $\mathfrak{S}_{24}=\mathfrak{S}_{\aaa,24}(L)\cap \mathfrak{S}_{\aaa,24}(K)=\mathfrak{S}_{\aaa,24}(L)$
and we may verify that $\mathcal{N}(\mathfrak{S}_{24})$ is the complete set of residues modulo $12$. Let us set $s=24$.

Note that $[9,b',a'] \ra \gen(L)$ if and only if 
$9a'-b'^2$ is not of the form $5^{2c+1}(5d\pm2)$ 
for some non-negative integers $c$ and $d$.
Moreover, if $(a',b')\Mod{24} \in \mathfrak{S}_{24}$ and $[9,b',a'] \ra \gen(L)$, then there is a representation $\sigma : [9,b',a'] \ra L$ with $\sigma(v_1)=w_\aaa$.

Let $N\ge N_\aaa=2373728$ be a positive integer, and let us set, $B=3$.
Let us choose any possible residue $r_{N,s}$ modulo $3s=72$.
Let $b_0\in\n$ be any integer such that $b_0 \equiv r_{N,s} \Mod{18}$.
For any integer $i$, we define $a_i$ and $b_i$ as in \eqref{defaibi}, and define
$D_i:=Aa_i-b_i^2$.
One may easily show that $D_k\not\equiv 0\Mod{5}$ for some $k\in\{0,6,12\}$.
Let $b=b_k$ for such $k$, and let $a=\frac{2}{3}(N-b)+b=a_k$.
Note that $9a-b^2 \not\equiv 0 \Mod{5}$, hence
$[9,b,a]\ra \gen(L)$. 

Thus, there is an integer $b\in I_{N,A}$ satisfying \eqref{condb:1} and \eqref{condb:2} since the length of $I_{N,A}$ is greater than or equal to $3sB=216$ by \eqref{condI}. Hence $N\pra P_{m,\aaa}$.\\
%


\noindent {\bf (Case 4)} $\aaa=(1,1,3,6)$, $A=A_\aaa=11$, $L=L_\aaa=\df{1,1,3,6}$.

This $\aaa$ is also of {\bf Type 3}, however we provide the proof of this case to show some argument that deals with a different kind of congruence conditions for the local representation. Note that the class number of $L$ is $2$, and all isometric classes in the genus are
$$
L=\df{1,1,3,6} \quad \text{and} \quad 
K=\begin{pmatrix}2&1 \\ 1&2\end{pmatrix}\perp \df{2,3}.
$$
As described in the above cases, we may verify that $\mathcal{N}(\mathfrak{S}_{8})=\{0,1,2,3\}$, so let us set $s=8$.

Note that $[11,b',a'] \ra \gen(L)$ if and only if 
$$
11a'-b'^2 \not \equiv 2\Mod{3} \quad \text{and} \quad [11,b',a'] \not \cong \df{11,11} \ \text{over} \ \q_{11},
$$
and the latter condition holds if $b' \not\equiv 0\Mod{11}$.
Moreover, if $(a',b')\Mod{8} \in \mathfrak{S}_{8}$ and $[11,b',a'] \ra \gen(L)$, then there is a representation $\sigma : [11,b',a'] \ra L$ with $\sigma(v_1)=w_\aaa$.

Let $N\ge N_\aaa=324897$ be a positive integer, and let us set $B=3$.
Let us choose any possible residue $r_{N,s}$ modulo $3s=12$, and
let $b_0\in\n$ be any integer such that $b_0 \equiv r_{N,s} \Mod{12}$.
For any integer $i$, we define $a_i$ and $b_i$ as in \eqref{defaibi}, and define
$D_i:=Aa_i-b_i^2$.
One may easily show that $\{ D_i \Mod{3} : i=0,4,8\}=\{0,1,2\}$, and at most one of $\{b_i : i=0,4,8\}$ is a multiple of $11$. Therefore, $D_k\not\equiv 2\Mod{3}$ and $b_k\not\equiv0 \Mod{11}$ for some $k\in\{0,4,8\}$.
Let $b=b_k$ for such $k$, and let $a=\frac{2}{3}(N-b)+b=a_k$.
Note that $11a-b^2\not\equiv 2 \Mod{3}$ and $b\not\equiv 0\Mod{11}$, hence
$[11,b,a]\ra \gen(L)$.

Thus, there is an integer $b\in I_{N,A}$ satisfying \eqref{condb:1} and \eqref{condb:2} since the length of $I_{N,A}$ is greater than or equal to $3sB=36$ by \eqref{condI}.
Hence $N\pra P_{m,\aaa}$.\\

\noindent {\bf (Case 5)} $\aaa=(1,2,3,8)$, $A=A_\aaa=14$, $L=L_\aaa=\df{1,2,3,8}$.

Note that this $\aaa$ is also of {\bf Type 3}, precisely, the class number of $L$ is $2$, and all isometric classes in the genus are
$$
L=\df{1,2,3,8} \quad \text{and} \quad 
K=\df{2}\perp\begin{pmatrix}3&1&1\\1&3&1\\1&1&4\end{pmatrix}.
$$
We may verify that 
$$
\mathfrak{S}_{12}\supseteq\mathfrak{S}_{12}':=\{(3,9),(11,7),(11,5),(3,3),(11,1),(11,11)\},
$$
and $\mathcal{N}(\mathfrak{S}_{12})=\{0,1,2,3,4,5\}$. Let us set $s=12$.

Note that $[14,b',a'] \ra \gen(L)$ if and only if
$9a'-b'^2$ is not of the form $2^{2c+2}(8d+5)$, $4(8d+3)$, or $4(8d+7)$ 
for some non-negative integers $c$ and $d$.
Moreover, this holds if $a'\equiv b' \equiv 1\Mod{2}$.
Therefore, if $(a',b') \Mod{24} \in\mathfrak{S}_{12}'$, then there is a representation $\sigma : [14,b',a'] \ra L$ with $\sigma(v_1)=w_\aaa$.

Let $N\ge N_\aaa=103969$ be a positive integer, and let us set $B=1$.
Note that the length of $I=I_{N,A}$ is greater than or equal to $3sB=36$ by \eqref{condI}.
Let us choose the residue $r_{N,s}$ modulo $3s=36$ corresponding to the element in $\mathfrak{S}_{12}'$ according as the residue of $N$ modulo $6$.
Let $b$ be a positive integer in $I$ such that $b \equiv r_{N,s} \Mod{36}$, and let $a=\frac{2}{3}(N-b)+b$.
Note that $(a,b) \Mod{12} \in \mathfrak{S}_{12}'$ by Lemma $\ref{complete-residue}$.
Thus, this integer $b$ satisfies \eqref{condb:1} and \eqref{condb:2}.
Hence $N\pra P_{m,\aaa}$.\\

\noindent {\bf (Case 6)} $\aaa=(1,2,3,3)$, $A=A_\aaa=9$, $L=L_\aaa=\df{1,2,3,3}$.

Note that the class number of $L$ is $2$ and all isometric classes in the genus are
$$
L=\df{1,2,3,3}, \quad \text{and} \quad
K=\df{1}\perp\begin{pmatrix}2&1 \\ 1&2\end{pmatrix}\perp \df{6}.
$$
Note that $[9,b',a'] \ra \gen(L)$ if and only if 
$9a'-b'^2$ is not of the form $3^{2c}(3d+1)$	
for some non-negative integers $c$ and $d$.

Let $N\ge 36689$ be a positive integer with $N\equiv 0\Mod{3}$.
We may verify that 
$$
\{ (\alpha,\beta)\in \mathfrak{S}_{\aaa,6} \mid \alpha \equiv \beta \Mod{2}\}=\{(0,0),(3,3)\}.
$$
Therefore, if $a'\equiv b' \equiv 0 \Mod{3}$ and $[9,b',a']\ra \gen(L)$ then there is a representation $\sigma : [9,b',a']\ra L \text{ with } \sigma(v_1)=w_\aaa$.
Note that the length of $I=I_{N,A}$ is greater than or equal to $27$.
Let $b\in I$ be a positive integer such that
$$
\begin{cases}
b \equiv 3 \Mod{9} & \text{if } N\equiv 3 \Mod{9},\\
b \equiv 6 \Mod{9} & \text{if } N\equiv 6 \Mod{9},\\
b \equiv 9 \Mod{27} & \text{if } N\equiv 0 \Mod{27},\\
b \equiv 18 \Mod{27} & \text{if } N\equiv 9 \Mod{27},\\
b \equiv 0 \Mod{27} & \text{if } N\equiv 18 \Mod{27}.\\
\end{cases}
$$
Then for $a=\frac{2}{3}(N-b)+b$, we have $a\equiv b\equiv 0\Mod{3}$ and $9a -b^2$ is not of the form $3^{2c}(3d+1)$ for some non-negative integers $c$ and $d$, since
$$
\begin{cases}
9a -b^2 \equiv 18 \Mod{27} &\text{if } N\equiv 3 \text{ or } 6\Mod{9},\\
9a -b^2 \equiv 27 \Mod{81} &\text{if } N\equiv0\Mod{9}.
\end{cases}
$$
Therefore, we have $[9,b,a]\ra\gen(L)$, so $b$ satisfies \eqref{condb:1} and \eqref{condb:2}. Hence $N\pra P_{5,\aaa}$.

Let $N\ge N_\aaa=262933$ be a positive integer with $N\not\equiv 0\Mod{3}$.
As before, we may verify that $\mathcal{N}(\mathfrak{S}_{24})$ is the complete set of residue modulo $12$, so let us set $s=24$.
Let us set $B=1$.
We choose any possible $r_{N,24}$ and let $b\in\n$ be any integer such that $b\equiv r_{N,24}\Mod{72}$.
Since $b\equiv N \Mod{3}$, we have $b\not\equiv 0 \Mod{3}$, so we have $9a-b^2\equiv 2\Mod{3}$, hence $[9,b,a]\ra\gen(L_\aaa)$.

Thus, there is an integer $b\in I_{N,A}$ satisfying \eqref{condb:1} and \eqref{condb:2} since the length of $I_{N,A}$ is greater than or equal to $3sB=72$ by \eqref{condI}.
Hence $N\pra P_{m,\aaa}$.\\

\noindent {\bf (Case 7)} $\aaa=(1,2,4,5)$, $A=A_\aaa=12$, $L=L_\aaa=\df{1,2,4,5}$.

Note that the class number of $L$ is $3$, and all isometric classes in the genus are
$$
L=\df{1,2,4,5}, \quad K_1=\df{1,1,2,20}, \quad \text{and} \quad K_2=\begin{pmatrix}2&1&1&0 \\ 1&3&1&1 \\ 1&1&3&1 \\ 0&1&1&4 \end{pmatrix}.
$$
Note that $[12,b',a']\ra \gen(L)$ if and only if $12a'-b'^2\neq 2^{2c}(16d+6)$ for some non-negative integers $c$ and $d$.

Let us set $s=24$. We may compute the set $\mathfrak{S}_{24}$ and may verify that $\mathcal{N}(\mathfrak{S}_{24})=\{0,3,4,6,7,8,9,11\}$. 
For each $n\in \mathcal{N}(\mathfrak{S}_{24})$, define 
$$
\mathfrak{S}_{24}^{(n)}=\left\{(\alpha,\beta) \in \mathfrak{S}_{24} \mid \alpha\equiv\beta \Mod{2}, \ \ \frac{3}{2}(\alpha-\beta)+\beta \equiv n \Mod{12}\right\}.
$$
Then we may observe that
\begin{gather}
\mathfrak{S}_{24}^{(0)}\supseteq\{(3,9)\}, \ \ \mathfrak{S}_{24}^{(3)}\supseteq\{(3,3)\}, \ \ \mathfrak{S}_{24}^{(6)}\supseteq\{(3,21)\}, \ \ \mathfrak{S}_{24}^{(9)}\supseteq\{(9,9)\}, \label{eq1245:0}\\ \nonumber
\mathfrak{S}_{24}^{(4)}=\{(0,16),(4,4),(8,16),(12,4),(16,16),(20,4)\}, \\ \nonumber
\mathfrak{S}_{24}^{(7)}=\{(2,16),(6,4),(10,16),(14,4),(18,16),(22,4)\},\\\nonumber
\mathfrak{S}_{24}^{(8)}=\{(0,8),(4,20),(8,8),(12,20),(16,8),(20,20)\},\\\nonumber
\mathfrak{S}_{24}^{(11)}=\{(2,8),(6,20),(10,8),(14,20),(18,8),(22,20)\},
\end{gather}

Let $N\ge 355873$ be a positive integer and let $B=1$. Note that the length of $I=I_{N,A}$ is greater than or equal to $72$. 
First assume that $N\equiv 0\Mod{3}$. Let us choose the residue $r_{N,s}$ modulo $3s=72$ corresponding to the $(\alpha,\beta)\in (\z/24\z)^2$ given in \eqref{eq1245:0} according as the residue of $N$ modulo $12$. Let $b\in I$ be a positive integer such that $b\equiv r_{N,s}\Mod{72}$, and let $a=\frac{2}{3}(N-b)+b$.
Note that $(a,b) \Mod{24} \in \mathfrak{S}_{24}$ by Lemma $\ref{complete-residue}$, so $a\equiv b \equiv 1 \Mod{2}$.
Therefore, $12a-b^2\equiv 1 \Mod{2}$, hence $[12,b,a]\ra \gen(L_\aaa)$.
Thus, this integer $b$ satisfies \eqref{condb:1} and \eqref{condb:2}.
Hence $N\pra P_{m,\aaa}$.

Now assume that $N\equiv4 \Mod{12}$.
Let us choose any possible residue $r_{N,s}$ modulo $72$, which corresponds to one of the pairs in $\mathfrak{S}_{24}^{(4)}$. 
Let $b_0\in\n$ be any integer such that $b_0 \equiv r_{N,s} \Mod{72}$.
For any integer $i$, we define $a_i$ and $b_i$ as in \eqref{defaibi}, and define
$D_i:=Aa_i-b_i^2$.
We claim that 
\begin{enumerate}[label={\rm(\arabic*)}]
	\item $(a_{4j},b_{4j}) \Mod{24} \in \mathfrak{S}_{24}^{(4)}$ for any integer $j$, and
	\item $\{D_i \Mod {64} : i=0,4,8,12\}=\{0,16,32,48\}$.
\end{enumerate}
The first assertion follows clearly from the definition of $(a_i,b_i)$ with the fact that 
$$
\mathfrak{S}_{24}^{(4)}=\{(4j,16+12j)\Mod{24} : j\in \z\} \quad  \text{and} \quad  (a_0,b_0) \Mod{24} \in \mathfrak{S}_{24}^{(4)}.
$$ 
To prove the second one, note that $D_{4j}\equiv0\Mod{16}$ for any integer $j$ since $a_{4j}\equiv b_{4j} \equiv 0 \Mod{4}$.
Moreover, one may show that $D_{4j_1} \equiv D_{4j_2} \Mod{64}$ if and only if $j_1 \equiv j_2 \Mod{16}$.
These implies the second assertion.
Thus, $D_k \equiv 16 \Mod{64}$ for some $k \in \{0,4,8,12\}$. Let $b=b_k$ for such $k$, and let $a=\frac{2}{3}(N-b)+b=a_k$.
Since $12a-b^2=D_k\equiv 16\Mod{64}$, $[12,b,a]\ra\gen(L)$. 
To sum up, we may find a integer $b$ among any $72$ consecutive integers satisfying \eqref{condb:2}.
Thus, there is a positive integer $b\in I$ satisfying \eqref{condb:1} and \eqref{condb:2}.
Hence $N\pra P_{5,\aaa}$.

The proofs for $N\equiv 7,8, \text{ or } 11\Mod{12}$ are quite similar to the above, so they are left to the reader.\\

\noindent {\bf (Case 8)} $\aaa=(1,2,4,7)$, $A=A_\aaa=14$, $L=L_\aaa=\df{1,2,4,7}$.

Note that the class number of $L$ is $5$, and all isometric classes in the genus are
$$
\begin{array}{c}
L=\df{1,2,4,7}, \quad K_1=\df{1,1,4,14},\quad K_2=\df{1,1,2,28},\\ [5pt]
K_3=\df{2}\perp \begin{pmatrix}3&1&1 \\ 1&3&1 \\ 1&1&4 \end{pmatrix}, \quad \text{and} \quad
K_4=\df{2}\perp \begin{pmatrix}2&0&1 \\ 0&4&2 \\ 1&2&5 \end{pmatrix}.
\end{array}
$$
For the simplicity of the notation, we will denote the corresponding basis of $K_i$ by $\{u_1,u_2,u_3,u_4\}$.
Note that we always have $[14,b',a'] \ra \gen(L)$.
We will construct a set $\mathfrak{S}\subset(\z/66\z)^2$ such that 
\begin{newenum}
	\item\label{cond1247:1} for any $(\alpha,\beta)\in\mathfrak{S}$, and for any binary $\z$-lattice $\ell$ in the set
	$$
	\left\{[14,b',a'] \mid (a',b') \equiv (\alpha,\beta) \Mod{66}, \ b'\not\equiv 0 \Mod{7}\right\},
	$$
	there exists a representation $\sigma:\ell\ra L$ such that $\sigma(v_1)=w_\aaa$.
	\item\label{cond1247:2} the set $\mathcal{N}(\mathfrak{S})$ is the complete set of residue modulo $33$.
\end{newenum}
Assume that $b'\not\equiv0\Mod{7}$. Since $\ell=[14,b',a'] \ra \gen(L)$, we have $\ell\ra L$ or $\ell\ra K_i$ for some $1\le i \le 4$ by Lemma \ref{localglobal}.

First, assume that $\rho : \ell\ra L$. Note that $|R(A,L)/O(L)|=1$, hence there is a $\tau\in O(L)$ such  that $\tau(\rho(v_1))=w_\aaa$.
Thus, $\sigma=\tau\circ\rho : \ell \ra L$ satisfies $\sigma(v_1)=w_\aaa$.

Next, assume that $\rho : \ell\ra K_1$. We may assume that $\rho(v_1)$ is equal to either
$$
u':=3u_1+u_2+u_3 \quad \text{or} \quad u'':=u_4.
$$
Let $\rho(v_2)=\sum_{i=1}^4 y_iu_i$ for some $y_i\in\z$. If $\rho(v_1)=u''$, then we have 
$$
b'=B(v_1,v_2)=B(\rho(v_1),\rho(v_2))=14y_4\equiv 0\Mod{7},
$$
which is a contradiction.
Therefore, we may assume that $\rho(v_1)=u'$. Let
$$
\mathfrak{S}(K_1)=\left\{ (\alpha,\beta)\in(\z/66\z)^2 \mid R_{u'}(K_1,\alpha,\beta,66) = R_{u'}^L(K_1,\alpha,\beta,66) \right\}.
$$
Then we have $\ell\ra L$ if $(a',b') \Mod{66} \in \mathfrak{S}(K_1)$.

Next, assume that $\rho : \ell\ra K_2$. Since $R(A,K_2)=\emptyset$, this is impossible.

Next, assume that $\rho : \ell\ra K_3$. 
Note that $\mathfrak{S}_{\aaa,2}(K_3)=\{(0,0),(0,1)\}$. Therefore, we have $\ell\ra K_3$ if $a'\equiv b'\equiv 0\Mod{2}$.

Finally, assume that $\rho : \ell\ra K_4$. We may assume that $\rho(v_1)$ is equal to
$$
u':=2u_1+u_2+u_3,  \quad u'':=u_1+2u_2+u_3, \quad \text{or} \quad u''':=u_2+u_3-2u_4.
$$
Let $\rho(v_2)=\sum_{i=1}^4 y_iu_i$ for some $y_i\in\z$. If $\rho(v_1)=u'''$, then we have 
$$
b'=B(v_1,v_2)=B(\rho(v_1),\rho(v_2))=-7y_4\equiv 0\Mod{7},
$$
which is a contradiction.
Therefore, we may assume that $\rho(v_1)=u'$ or $u''$. Let
$$
\mathfrak{S}(K_4)=\left\{ (\alpha,\beta)\in(\z/66\z)^2 \mid R_{u}(K_4,\alpha,\beta,66) = R_{u}^L(K_4,\alpha,\beta,66) \ \forall u\in\{u',u''\}\right\}.
$$
Then we have $\ell\ra L$ if $(a',b') \Mod{66} \in \mathfrak{S}(K_4)$.

Now, let us define
$$
\mathfrak{S}=\left\{(\alpha,\beta)\in \mathfrak{S}(K_1)\cap\mathfrak{S}(K_4) \mid \alpha\equiv\beta \equiv 0 \Mod{2}\right\} \subset (\z/66\z)^2.
$$
Then $\mathfrak{S}$ satisfies condition \ref{cond1247:1} by the above argument, and we may directly check that condition \ref{cond1247:2} also holds.

Let us set $s=66$ and $B=2$.
Let $N\ge N_\aaa=12687051$ be a positive integer.
Let us choose any possible residue $r_{N,s}$ modulo $3s=198$, and
let $b_0\in\n$ be any integer such that $b_0 \equiv r_{N,s} \Mod{198}$.
Note that $b_k=b_0+3k\not \equiv 0 \Mod{7}$ for some $k\in\{0,66\}$.
Then $b=b_k$ for such $k$ satisfies \eqref{condb:2}. 

Thus, there is an integer $b\in I_{N,A}$ satisfying \eqref{condb:1} and \eqref{condb:2} since the length of $I_{N,A}$ is greater than or equal to $3sB=396$ by \eqref{condI}.
Hence $N\pra P_{m,\aaa}$.
\end{proof}

\section{Universal sums of pentagonal numbers}\label{sec-escalation}

Recall that for $\aaa=(a_1,\ldots,a_k) \in \n^k$, and $\bm{x}=(x_1,\ldots,x_k)\in \z^k$, we say the sum
$$
P_{5,\aaa}(\bm{x}):=\sum_{i=1}^k a_i P_5(x_i)
$$
is {\em universal} over $\nn$ if $P_{5,\aaa}(\bm{x})=N$ has a solution $\bm{x}\in \n_0^k$ (in other words, $N\pra P_{5,\aaa}$) for any non-negative integer $N$.
If this is the case, we call the sum $P_{5,\aaa}$ a {\em universal sum of pentagonal numbers}. 

A universal sum of pentagonal numbers $P_{5,\aaa}$ is called {\em proper} if  any proper sum $P_{5,\hat{\aaa}_i}$ $(1\le i \le k)$ of $P_{5,\aaa}$  is not universal over $\nn$, where $\hat{\aaa}_i=(a_1,\ldots,\hat{a}_i,\ldots,a_k)$.
Here, the `hat' symbol $\, \hat{} \, $ over $a_i$ indicate that this integer is deleted from the sequence $a_1,\ldots,a_n$.
When the sum $P_{5,\aaa}$ is not universal over $\nn$, the least positive integer that is not represented by $P_{5,\aaa}$ over $\nn$ is called the {\em truant} of $P_{5,\aaa}$ and is denoted by $t(P_{5,\aaa})$.
Without loss of generality, we always assume that 
$$
a_1\leq a_2\leq \cdots \leq a_k.
$$

In this section, we determine all proper universal sums of pentagonal numbers by using the escalation method (see Table \ref{table3} for the complete list). Furthermore, we give the proof of Theorem \ref{63thm} which provides an effective criterion on the universality of an arbitrary sum of pentagonal numbers. 

For $\aaa=(a_1,a_2,\dots,a_k)\in\n^k$, we define $\aaa_i:=(a_1,a_2,\dots,a_i)\in\n^i$ for each $1\leq i\leq k$.
In proving our results, we will frequently use the following lemma.

\begin{lem}\label{lem-bound}
Let $k$ and $N$ be positive integers and let $\aaa=(a_1,\ldots,a_k)\in\n^k$ with $a_1\le\cdots \le a_k$. Assume that $N$ is represented by $P_{5,\aaa}$ over $\nn$. Then, we have $a_1\le N$.
Moreover, if $N$ is not represented by $P_{5,\aaa_{i-1}}$ over $\nn$ for some $i\in\{2,\ldots,k\}$, then $a_i\le N$.
\end{lem}
\begin{proof}
Note that $\{P_5(x) \mid x \in \nn\}=\{0, 1, 5, 12, 22, 35, 51,\ldots\}$. Indeed, these integers are the least seven pentagonal numbers.
Assume that 
$$
P_{5,\aaa}(\bm{x})=\sum_{i=1}^k a_i P_5(x_i)=N
$$ 
for some $\bm{x}=(x_1,\ldots,x_k)\in \n_0^k$.
Assume that $a_i>N$ for some $2\le i \le k$. 
Then, we should have $x_j=0$ for any $i\le j \le k$ since otherwise, we have $P_{5,\aaa}(\bm{x}) >N$ because $N<a_i\le\cdots\le a_k$.
Hence, we have $N\pra P_{5,\aaa_{i-1}}$.
One may show that if $a_1>N$, then $P_{5,\aaa}(\bm{x})=N$ cannot have a solution $\bm{x}\in \n_0^k$. This proves the lemma.
\end{proof}

Now, we determine all proper universal sums of pentagonal numbers.

Assume that a sum $P_{5,\aaa}$ is proper universal over $\nn$. 
Since $1\pra P_{5,\aaa}$, $a_1=1$ by Lemma \ref{lem-bound}.
Since $2\pra P_{5,\aaa}$ and $t(P_{5,(1)})=2$, we have $1\le a_2\le 2$ by Lemma \ref{lem-bound}.
Note that 
$$
t(P_{5,\aaa_2})=
\begin{cases}
3 &\text{ if } \aaa_2=(1,1),\\
4 &\text{ if } \aaa_2=(1,2).
\end{cases}
$$
Therefore, by Lemma \ref{lem-bound}, $\aaa_3$ should be one of the following
$$
(1,1,a_3) \text{ with } 1\le a_3 \le 3, \quad  (1,1,a_3) \text{ with } 2\le a_3 \le 4.
$$
One may easily check that there are no ternary universal sums of pentagonal numbers.
Indeed, for each of the above possible $\aaa_3$ the truant $t(P_{5,\aaa_3})$ of $P_{5,\aaa_3}$ is 
$$
t(P_{5,\aaa_3})=
\begin{cases}
4 &\text{ if } \aaa_3=(1,1,1),\\
9 &\text{ if } \aaa_3=(1,1,2),\\
7 &\text{ if } \aaa_3=(1,1,3),\\
6 &\text{ if } \aaa_3=(1,2,2),\\
9 &\text{ if } \aaa_3=(1,2,3),\\
8 &\text{ if } \aaa_3=(1,2,4).
\end{cases}
$$
Again by Lemma \ref{lem-bound}, we have $a_3\leq a_4\leq t(P_{5,\aaa_3})$ for each of the possible cases.
Therefore, there are $34$ candidates of quaternary sums $P_{5,\aaa_4}$ of pentagonal numbers.
One may easily check that exactly $15$ quaternary sums $P_{5,\aaa_4}$ among those $34$ candidates are universal over $\nn$ by Theorem \ref{thmquatpolyrep} (see also Tables \ref{table-data1} and \ref{table-data2}).

Assume that $P_{5,\aaa_4}$ is not universal over $\nn$. Then, $\aaa_4$ should be one of the $19=34-15$ quadruples in Table \ref{table-data1} or \ref{table-data2} other than $(1,2,4,12)$ such that the set 
$$
E(P_{5,\aaa})=\nn\setminus\{n\in\nn \mid  n \pra P_{5,\aaa}\}
$$ 
for $\aaa=\aaa_4$ is a nonempty set.
Note that for each of those $\aaa_4$'s, $t(P_{5,\aaa_4})$ is the least integer in the set $E(P_{5,\aaa_4})$.
Moreover, we necessarily have $a_4\leq a_5\leq t(P_{5,\aaa_4})$ by Lemma \ref{lem-bound}. 
Therefore, there are $373$ candidates of quinary sums $P_{5,\aaa_5}$.

We claim that exactly $371$ quinary sums among $373$ candidates, namely those except two quinary sums $P_{5,\aaa}$ with $\aaa=(1,1,3,4,7)$ and $(1,2,2,5,10)$, are in fact universal over $\nn$
(see the second part of Table \ref{table3} for the complete list of quinary proper universal sums of pentagonal numbers).
For the quinary sums with $\aaa_4 \neq (1,2,4,5)$ among those $371$ quinary sums, 
one may easily prove that they are universal over $\nn$ from the set $E(P_{5,\aaa_4})$ given in Table \ref{table-data1} or \ref{table-data2}.
Now it suffice to show that the quinary sums $P_{5,\aaa}$ with
$$
\aaa = (1,2,4,5,a_5) \quad \text{with} \quad 5\le a_5 \le 13
$$
are universal over $\nn$. In the case when $a_5=12$, the sum $P_{5,(1,2,4,5,12)}$ is universal over $\nn$ since $E(P_{5,(1,2,4,12)})=\{8,138\}$ by Theorem \ref{thmquatpolyrep}.

Now assume that $5\le a_5 \le 13$ with $a_5\neq12$. 
Note that $N\pra P_{5,(1,2,4,5)}$ for any integer $N\ge 355873$ with $N \not \equiv 1,2,5,10 \Mod{12}$ by Theorem \ref{thmquatpolyrep}, as well as any non-negative integer $N\le 10^7$ except $13$ (see Remark \ref{rmk1} \ref{rmk1:1}). Note that $13 \pra P_{5,(1,2,4,5,a_5)}$.
If $N\ge 10^7$ with $N\equiv 1,2,5, \text{ or }10\Mod{12}$, then one may easily check that
$$
N-a_5 P_5(x_5) \not\equiv 1,2,5,10 \Mod{12}
$$
for some $0\le x_5\le 4$. Since $N-a_5 P_5(x_5)\ge 10^7-13\cdot22\ge 355873$ for such $x_5$, we have $N-a_5P_5(x_5)\pra P_{5,(1,2,4,5)}$, hence $N\pra P_{5,(1,2,4,5,a_5)}$. Thus $P_{5,(1,2,4,5,a_5)}$ is universal over $\nn$.
This completes the proof of the claim on quinary sums.

Assume that $P_{5,\aaa_5}$ is not universal over $\nn$. 
Then we have
$$
\aaa_5=(1,1,3,4,7) \quad  \text{or} \quad (1,2,2,5,10).	
$$
We first consider the case when $\aaa_5=(1,1,3,4,7)$. 
Since $E(P_{5,(1,1,3,4)})=\{11,18,81\}$, the quinary sum $P_{5,(1,1,3,4,7)}$ represents all positive integers over $\nn$ except for $18$.
Therefore, $P_{5,(1,1,3,4,7,a_6)}$ is universal over $\nn$ for any $7\leq a_6\leq 18$. Indeed, by Lemma \ref{lem-bound}, these $a_6$'s are the only possible integers for which $P_{5,(1,1,3,4,7,a_6)}$ are universal over $\nn$.
Similarly, if $\aaa_5=(1,2,2,5,10)$, then one may show that $P_{5,(1,2,2,5,10,a_6)}$ is universal over $\nn$ if and only if $10\leq a_6\leq 33$. 
Thus, there are $36$ senary proper universal sums of pentagonal numbers.
This completes the escalation.\\

Now, we are ready to prove Theorem \ref{63thm}

\begin{proof}[Proof of Theorem \ref{63thm}]
For $\aaa=(a_1,a_2,\dots,a_k)\in\n^k$, assume that the sum $P_{5,\aaa}$ of pentagonal numbers represents the integers
$$
\text{$1$, $2$, $3$, $4$, $6$, $7$, $8$, $9$, $11$, $13$, $14$, $17$, $18$, $19$, $23$, $28$, $31$, $33$, $34$, $39$, $42$, and $63$.}
$$
Repeating a similar argument used to classify proper universal sums of pentagonal numbers, one may easily show that $P_{5,\aaa_i}~(4\leq i\leq6)$ is one of the sums listed in Table \ref{table3}. 
Since the proper sum $P_{5,\aaa_i}$ of $P_{5,\aaa}$ is universal over $\nn$, the sum $P_{5,\aaa}$ of pentagonal numbers is universal over $\nn$. 
\end{proof}

\newpage

\begin{table}[!hp]
	\caption{Data for the proof of Theorem \ref{thmquatpolyrep} I} 
	\label{table-data1}
	\begin{center}
		\small
		\begin{tabular}{|c|c|c|c|c|c|}
			\hline
			
			\multirow{2}{*}{$\aaa$} & \multirow{2}{*}{Type} & {$N_\aaa$} & {$s_\aaa$} & {$B_\aaa$} & Used sufficient condition for\\ \cline{3-5}
			&& \multicolumn{3}{c|}{$E(P_{5,\aaa})$} & $[A_\aaa,b',a']\ra \gen(L_\aaa)$ $(\ast)$\\ 	\hline \hline
			
			\multirow{2}{*}{$(1,1,1,1)$}& \multirow{2}{*}{2} & $711$ & $2$ & $1$ &  \multirow{2}{*}{$a'\equiv b'\equiv 1\Mod{2}$}\\  \cline{3-5}
			&& \multicolumn{3}{c|}{$\{9,21,31,43,55,89\}$}& \\ \hline\hline
			
			\multirow{2}{*}{$(1,1,1,2)$}& \multirow{2}{*}{2} & $3769$ & $2$ & $2$ &  \multirow{2}{*}{$b'\not\equiv 0\Mod{5}$}\\ \cline{3-5}
			&& \multicolumn{3}{c|}{$\emptyset$}& \\	\hline\hline			           
			
			\multirow{2}{*}{$(1,1,1,3)$}& \multirow{2}{*}{2} & $10473$ & $2$ & $3$ &  \multirow{2}{*}{$a'\equiv 2\Mod{3}$}\\ \cline{3-5}
			&& \multicolumn{3}{c|}{$\{19\}$}& \\ \hline\hline
			
			\multirow{2}{*}{$(1,1,1,4)$}& \multirow{2}{*}{1} & $5453$ & $1$ & $4$ &  \multirow{2}{*}{$a' \equiv 2 \Mod{4}$}\\ \cline{3-5}
			&& \multicolumn{3}{c|}{$\{8\}$}& \\	\hline\hline
			
			\multirow{2}{*}{$(1,1,2,2)$}& \multirow{2}{*}{2} & $1120$ & $2$ & $1$ &  \multirow{2}{*}{$a'\equiv b'\equiv 1\Mod{2}$}\\ \cline{3-5}
			&& \multicolumn{3}{c|}{$\emptyset$}& \\  \hline\hline
			
			\multirow{2}{*}{$(1,1,2,3)$}& \multirow{2}{*}{2} & $5453$ & $2$ & $2$ &  \multirow{2}{*}{$ b'\not\equiv 0 \Mod{7}$}\\ \cline{3-5}
			&& \multicolumn{3}{c|}{$\emptyset$}& \\	\hline\hline
			
			\multirow{2}{*}{$(1,1,2,4)$}& \multirow{2}{*}{2} & $6294$ & $4$ & $1$ &  \multirow{2}{*}{$a'\equiv b'\equiv 1\Mod{2}$}\\ \cline{3-5}
			&& \multicolumn{3}{c|}{$\emptyset$}& \\	\hline\hline
			
			\multirow{2}{*}{$(1,1,2,5)$}& \multirow{2}{*}{3} & $2373728$ & $24$ & $3$ & \multirow{2}{*}{$9a'-b'^2 \not\equiv 0 \Mod{5}$}\\ \cline{3-5}
			&& \multicolumn{3}{c|}{$\emptyset$}& \\	\hline\hline
			
			\multirow{2}{*}{$(1,1,2,6)$}& \multirow{2}{*}{2} &  $73138$& $6$ & $2$ &  \multirow{2}{*}{ $b'\not\equiv 0\Mod{5}$}\\ \cline{3-5}
			&& \multicolumn{3}{c|}{$\emptyset$}& \\	\hline\hline
			
			\multirow{2}{*}{$(1,1,2,7)$}& \multirow{2}{*}{3} & $324897$  & $24$ & $1$ & \multirow{2}{*}{Any $a'$ and $b'$}\\ \cline{3-5}
			&& \multicolumn{3}{c|}{$\{28\}$}& \\	\hline\hline
			
			\multirow{2}{*}{$(1,1,2,8)$}& \multirow{2}{*}{3} & $88555$ & $12$ & $1$ & \multirow{2}{*}{$a'\equiv b' \equiv 1 \Mod{2} $}\\ \cline{3-5}
			&& \multicolumn{3}{c|}{$\{39\}$}& \\	\hline\hline
			
			\multirow{2}{*}{$(1,1,2,9)$}& \multirow{2}{*}{3} & $1550980$ & $24$ & $2$ & \multirow{2}{*}{$b'\not\equiv 0 \Mod{13}$}\\ \cline{3-5}
			&& \multicolumn{3}{c|}{$\{18\}$}& \\	\hline\hline
			
			\multirow{2}{*}{$(1,1,3,3)$}& \multirow{2}{*}{2} & $14295$ & $2$ & $3$ & {$8a'-b'^2 \equiv 1 \Mod{3}$}\\ \cline{3-5}
			&& \multicolumn{3}{c|}{$\{14\}$}& and $a'\equiv b' \equiv 1 \Mod{2}$ \\	\hline\hline
			
			\multirow{2}{*}{$(1,1,3,4)$}& \multirow{2}{*}{3} & $65428$ & $12$ & $1$ & \multirow{2}{*}{$a'\equiv 2 \Mod{3}$ if $b\equiv0\Mod{3}$}\\ \cline{3-5}
			&& \multicolumn{3}{c|}{$\{11,18,81\}$}& \\	
			\hline\hline
			
			\multirow{2}{*}{$(1,1,3,5)$}& \multirow{2}{*}{3} & $73138$ & $6$ & $2$ & \multirow{2}{*}{$b'\not\equiv 0 \Mod{5}$}\\ \cline{3-5}
			&& \multicolumn{3}{c|}{$\{19\}$}& \\	
			\hline\hline
			
			\multirow{2}{*}{$(1,1,3,6)$}& \multirow{2}{*}{3} & $324897$  & $8$ & $3$ & {$11a'-b'^2\not\equiv2\Mod{3}$}\\ \cline{3-5}
			&& \multicolumn{3}{c|}{$\emptyset$}& and $b'\not\equiv 0 \Mod{11}$\\	\hline\hline
			
			\multirow{2}{*}{$(1,1,3,7)$}& \multirow{2}{*}{3} & $801954$ & $12$ & $3$ & \multirow{2}{*}{$12a'-b'^2\not \equiv 0 \Mod{7}$}\\ \cline{3-5}
			&& \multicolumn{3}{c|}{$\{14,18\}$}& \\
			
			\hline
		\end{tabular}
	\end{center}
	{\raggedright \tiny $(\ast)$ $a'\in\n$ and $b'\in\z$ are such that $a'\equiv b'\Mod{2}$ and $A_\aaa a'-b'^2>0$.}
\end{table}

\newpage 

\begin{table}[!hp]
	\caption{Data for the proof of Theorem \ref{thmquatpolyrep} II} 
	\label{table-data2}
	\begin{center}
		\small
		\begin{tabular}{|c|c|c|c|c|c|}
			\hline
			
			\multirow{2}{*}{$\aaa$} & \multirow{2}{*}{Type} & {$N_\aaa$} & {$s_\aaa$} & {$B_\aaa$} & Used sufficient condition for \\ \cline{3-5}
			&& \multicolumn{3}{c|}{$E(P_{5,\aaa})$} & $[A_\aaa,b',a']\ra \gen(L_\aaa)$ $(\ast\ast)$\\ \hline \hline

			\multirow{2}{*}{$(1,2,2,2)$}& \multirow{2}{*}{1} & $16902$ & $1$ & $7$ & \multirow{2}{*}{$7a'-b'^2\neq 2^{2c}(16d+14)$}\\ \cline{3-5}				
			&& \multicolumn{3}{c|}{$\{8,43,67,135\}$}& \\	\hline\hline
			
			\multirow{2}{*}{$(1,2,2,3)$}& \multirow{2}{*}{2} & $1529$ & $2$ & $1$ & \multirow{2}{*}{$a'\equiv b'\equiv 1\Mod{2}$}\\ \cline{3-5}				
			&& \multicolumn{3}{c|}{$\{33\}$}& \\	\hline\hline
			
			\multirow{2}{*}{$(1,2,2,4)$}& \multirow{2}{*}{2} & $1054175$ & $12$ & $4$ & \multirow{2}{*}{$9a'-b'^2\neq 2^{2c}(8d+7)$}\\ \cline{3-5}				
			&& \multicolumn{3}{c|}{$\emptyset$}& \\	\hline\hline
			
			\multirow{2}{*}{$(1,2,2,5)$}& \multirow{2}{*}{3} & $32354$ & $8$ & $1$ & \multirow{2}{*}{Any $a'$ and $b'$}\\ \cline{3-5}				
			&& \multicolumn{3}{c|}{$\{23,33\}$}& \\	\hline\hline
			
			\multirow{2}{*}{$(1,2,2,6)$}& \multirow{2}{*}{2} & $1302606$ & $24$ & $2$ & \multirow{2}{*}{$b'\not\equiv 0\Mod{11}$}\\ \cline{3-5}				
			&& \multicolumn{3}{c|}{$\emptyset$}& \\	\hline\hline
			
			\multirow{2}{*}{$(1,2,3,3)$}& \multirow{2}{*}{4} & $262933$ & $24$ & $1$ & \multirow{2}{*}{$9a'-b'^2\neq 3^{2c}(3d+1)$}\\ \cline{3-5}				
			&& \multicolumn{3}{c|}{$\{34,99\}$}& \\	\hline\hline

			\multirow{2}{*}{$(1,2,3,4)$}& \multirow{2}{*}{3} &  $18114$ & $6$ & $1$ & \multirow{2}{*}{ $a'\equiv b'\equiv 1 \Mod{2}$}\\ \cline{3-5}
			&& \multicolumn{3}{c|}{$\emptyset$}& \\  \hline\hline
			
			\multirow{2}{*}{$(1,2,3,5)$}& \multirow{2}{*}{3} & $1302606$ & $24$ & $2$ &  \multirow{2}{*}{$b'\not\equiv 0\Mod{11}$}\\ \cline{3-5}				
			&& \multicolumn{3}{c|}{$\emptyset$}& \\	\hline\hline
			
			\multirow{2}{*}{$(1,2,3,6)$}& \multirow{2}{*}{3} & $1426798$ & $24$ & $2$ & \multirow{2}{*}{$12a'-b'^2 \neq 2^{2c}(8d+7)$}\\ \cline{3-5}				
			&& \multicolumn{3}{c|}{$\{63\}$}& \\	\hline\hline
			
			\multirow{2}{*}{$(1,2,3,7)$}& \multirow{2}{*}{3} & $23841$ & $6$ & $1$ & \multirow{2}{*}{Any $a'$ and $b'$}\\ \cline{3-5}				
			&& \multicolumn{3}{c|}{$\emptyset$}& \\	\hline\hline
			
			\multirow{2}{*}{$(1,2,3,8)$}& \multirow{2}{*}{3} & $103969$ & $12$ & $1$ & \multirow{2}{*}{$a'\equiv b'\equiv 1 \Mod{2}$}\\ \cline{3-5}				
			&& \multicolumn{3}{c|}{$\{31,63\}$}& \\	\hline\hline
			
			\multirow{2}{*}{$(1,2,3,9)$}& \multirow{2}{*}{3} & $7205634$ & $32$ & $3$ & \multirow{2}{*}{$a'\equiv 0\Mod{3}$ if $b'\equiv 0\Mod{3}$}\\ \cline{3-5}				
			&& \multicolumn{3}{c|}{$\{42\}$}& \\	\hline\hline
			
			\multirow{2}{*}{$(1,2,4,4)$}& \multirow{2}{*}{2} & $578265$ & $16$ & $2$ & {$11a'-b'^2\neq 16d+6$}\\ \cline{3-5}				
			&& \multicolumn{3}{c|}{$\{17\}$}&and $b'\not\equiv0\Mod{11}$ \\	\hline\hline
			
			\multirow{2}{*}{$(1,2,4,5)$}& \multirow{2}{*}{4} & $355873$ & $24$ & $1$ & \multirow{2}{*}{$12a'-b'^2\neq 2^{2c}(16d+6)$}\\ \cline{3-5}				
			& &\multicolumn{3}{c|}{$\{13\}$ $(\ast\ast\ast)$}& \\	\hline\hline
			
			\multirow{2}{*}{$(1,2,4,6)$}& \multirow{2}{*}{2} & $1550980$ & $12$ & $4$ & \multirow{2}{*}{$13a'-b'^2\neq 2^{2c}(8d+5)$}\\ \cline{3-5}				
			&& \multicolumn{3}{c|}{$\emptyset$}& \\	\hline\hline
			
			\multirow{2}{*}{$(1,2,4,7)$}& \multirow{2}{*}{4} & $12687051$ & $66$ & $2$ & \multirow{2}{*}{Any $a'$ and $b'$}\\ \cline{3-5}				
			&& \multicolumn{3}{c|}{$\emptyset$}& \\	\hline\hline
			
			\multirow{2}{*}{$(1,2,4,8)$}& \multirow{2}{*}{3} & $1799322$ & $16$ & $3$ & \multirow{2}{*}{$15a'-b'^2\neq 2^{2c}(8d+7)$}\\ \cline{3-5}				
			&& \multicolumn{3}{c|}{$\emptyset$}& \\ \hline\hline
			
			\multirow{2}{*}{$(1,2,4,12)$}& \multirow{2}{*}{4} & $2295950$ & $24$ & $2$ & \multirow{2}{*}{$19a'-b'^2 \neq 2^{2c}(16d+10)$}\\ \cline{3-5}				
			&& \multicolumn{3}{c|}{$\{8, 138\}$}& \\				
			\hline
		\end{tabular}
	\end{center}
	{\raggedright \tiny 
	$(\ast\ast)$ $a'\in\n$ and $b'\in\z$ are such that $a'\equiv b'\Mod{2}$ and $A_\aaa a'-b'^2>0$, and $c,d$ are non-negative integers.\\[-3pt]
	$(\ast\ast\ast)$ Determining the $E(P_{5,(1,2,4,5)})$ still remains open, while $\{N\in E(P_{5,\aaa} \mid N\le 10^7\}=\{13\}$ is verified.}
\end{table}
\newpage 
\begin{table}[!hp]
	\caption{Proper universal sums $P_{5,\aaa}$ of pentagonal numbers} 
	\label{table3}
	\begin{center}
		\small
		\begin{tabular}{|c|c|}
			\hline
			$\aaa$ & Conditions on $a_k$ $(4\leq k\leq 6)$\\
			\hline\hline 
			$(1,1,1,a_4)$ & $a_4=2$\\
			$(1,1,2,a_4)$ & $2\leq a_4\leq 6$\\
			$(1,1,3,a_4)$ & $a_4=6$\\
			$(1,2,2,a_4)$ & $a_4=4,6$\\
			$(1,2,3,a_4)$ & $a_4=4,5,7$\\
			$(1,2,4,a_4)$ & $a_4=6,7,8$\\
			\hline\hline
			$(1,1,1,1,a_5)$ & $1\leq a_5\leq 9$\\
			$(1,1,1,3,a_5)$ & $3\leq a_5\leq 19$\\
			$(1,1,1,4,a_5)$ & $4\leq a_5\leq 8$\\
			$(1,1,2,7,a_5)$ & $7\leq a_5\leq 28$\\
			$(1,1,2,8,a_5)$ & $8\leq a_5\leq 39$\\
			$(1,1,2,9,a_5)$ & $9\leq a_5\leq 18$\\
			$(1,1,3,3,a_5)$ & $3\leq a_5\leq 14$\\
			$(1,1,3,4,a_5)$ & $4\leq a_5\leq 18$, $a_5\neq 7$\\
			$(1,1,3,5,a_5)$ & $5\leq a_5\leq 19$\\
			$(1,1,3,7,a_5)$ & $7\leq a_5\leq 14$\\
			$(1,2,2,2,a_5)$ & $2\leq a_5\leq 8$\\
			$(1,2,2,3,a_5)$ & $3\leq a_5\leq 33$\\
			$(1,2,2,5,a_5)$ & $5\leq a_5\leq 23$, $a_5\neq10$\\
			$(1,2,3,3,a_5)$ & $3\leq a_5\leq 34$\\
			$(1,2,3,6,a_5)$ & $6\leq a_5\leq 63$\\
			$(1,2,3,8,a_5)$ & $8\leq a_5\leq 31$\\
			$(1,2,3,9,a_5)$ & $9\leq a_5\leq 42$\\
			$(1,2,4,4,a_5)$ & $4\leq a_5\leq 17$\\
			$(1,2,4,5,a_5)$ & $5\leq a_5\leq 13$\\
			\hline\hline
			$(1,1,3,4,7,a_6)$ & $7\leq a_6\leq 18$\\
			$(1,2,2,5,10,a_6)$ & $10\leq a_6\leq 33$\\
			
			\hline
		\end{tabular}
	\end{center}
\end{table}

\end{document}